\documentclass[a4paper]{article}
\usepackage{color}
\usepackage{enumerate, enumitem}
\usepackage{amsthm,amsmath}
\newtheorem{theorem}{Theorem}
\newtheorem{definition}[theorem]{Definition}
\newtheorem{corollary}[theorem]{Corollary}
\newtheorem{lemma}[theorem]{Lemma}
\newtheorem{assumption}[theorem]{Assumption}
\newtheorem{remark}[theorem]{Remark}
\newtheorem{example}[theorem]{Example}
\numberwithin{theorem}{section}
\usepackage{hyperref}
\usepackage{cleveref}

\usepackage[numbers]{natbib}
\usepackage{algpseudocode}
\usepackage{bm}
\usepackage{nicematrix}
\usepackage{algorithm}
\usepackage{pifont}
\usepackage{csvsimple}
\usepackage{mathrsfs} 
\usepackage{titlesec}
\usepackage{multicol, multirow, tabularx}
\usepackage{mathtools, calc}
\usepackage{verbatim}
\usepackage{listings}
\usepackage{graphicx,epstopdf} 
\usepackage[caption=false]{subfig}
\usepackage{lscape, textcomp}
\usepackage{tikz}
\usetikzlibrary{tikzmark, calc}
\usepackage{diagbox}
\usepackage{xparse}

\usepackage{url}
\usepackage{enumitem}
\usepackage[tableposition=top]{caption}
\usepackage[toc,page]{appendix}
\usepackage[british]{babel}
\usepackage{tkz-euclide}
\usetikzlibrary{arrows}
\usepackage{fancyhdr} 
\usepackage{array, stackengine}
\usepackage{pbox}
\usepackage{tablefootnote, longtable}
\usepackage{dsfont}
\usepackage{makecell}

\newcolumntype{L}[1]{>{\raggedright\let\newline\\\arraybackslash\hspace{0pt}}m{#1}}
\newcolumntype{C}[1]{>{\centering\let\newline\\\arraybackslash\hspace{0pt}}m{#1}}
\newcolumntype{R}[1]{>{\raggedleft\let\newline\\\arraybackslash\hspace{0pt}}m{#1}}

\definecolor{blue}{rgb}{0.0, 0.0,1.0}
\definecolor{gray}{rgb}{0.7, 0.7, 0.7}
\definecolor{blue-violet}{rgb}{0.54, 0.17, 0.89}

\newcommand{\mtx}[1]{\boldsymbol{#1}}
\newcommand{\mvec}[1]{\boldsymbol{#1}}

\newcommand\bT{\mathcal T}
\newcommand\bR{\mathbb R}

\newcommand\set{\mathcal S}
\newcommand\func{\mathcal F}

\newcommand\calT{\mathcal{T}}

\newcommand\bp{\boldsymbol p}
\newcommand\bu{\boldsymbol u}
\newcommand\bv{\boldsymbol v}

\newcommand\bx{\boldsymbol x}
\newcommand\by{\boldsymbol y}
\newcommand\bz{\boldsymbol z}
\newcommand\ba{\boldsymbol a}
\newcommand\bC{\boldsymbol C}
\newcommand\bW{\boldsymbol W}
\newcommand\bA{\boldsymbol A}

\newcommand\bQ{\boldsymbol Q}

\newcommand\bU{\boldsymbol U}
\newcommand\bV{\boldsymbol V}

\def \bzero {\mathbf{0}}

\DeclareMathOperator{\trace}{trace}
\DeclareMathOperator{\rank}{rank}

\DeclareMathOperator{\range}{range}
\DeclareMathOperator{\prob}{\mathbb{P}}

\DeclarePairedDelimiter{\ceil}{\lceil}{\rceil}

\usepackage{amsfonts}
\def \R {\mathbb{R}}

\def \red {\mathrm{red}}

\def \spanvec {\mathrm{span}}

\usepackage{adjustbox}
\newcolumntype{R}[2]{%
    >{\adjustbox{angle=#1,lap=\width-(#2)}\bgroup}%
    l%
    <{\egroup}%
}

\newsavebox\CBox

\makeatletter
\newcommand{\subalign}[2][c]{%
	\if#1c\vcenter\else\vtop\fi{%
		\Let@ \restore@math@cr \default@tag
		\baselineskip\fontdimen10 \scriptfont\tw@
		\advance\baselineskip\fontdimen12 \scriptfont\tw@
		\lineskip\thr@@\fontdimen8 \scriptfont\thr@@
		\lineskiplimit\lineskip
		\ialign{\hfil$\m@th\scriptstyle##$&$\m@th\scriptstyle{}##$\hfil\crcr
			#2\crcr
		}%
	}%
}
\makeatother

\makeatletter
\renewcommand*{\@fnsymbol}[1]{\ensuremath{\ifcase#1\or *\or \ddagger\or \mathsection\or \vee\or \wedge\or \dagger\or
		\mathsection\or \mathparagraph\or \|\or **\or \dagger\dagger
		\or \ddagger\ddagger \else\@ctrerr\fi}}

\numberwithin{equation}{section}

\begin{document}

\title{Learning the subspace of variation for global optimization of functions with low effective dimension}
\author{
	Coralia Cartis\thanks{Mathematical Institute, University of Oxford, Radcliffe Observatory Quarter, Woodstock Road, Oxford, OX2 6GG, UK; The Alan Turing Institute, The British Library, London, NW1 2DB, UK, \texttt{cartis@maths.ox.ac.uk}. This work was supported by The Alan Turing Institute under The Engineering and Physical Sciences Research Council (EPSRC) grant EP/N510129/1 and under the Turing project scheme.} 
	\and Xinzhu Liang\thanks{The University of Manchester, Oxford Road, Manchester, M13 9PL, UK, \texttt{xinzhu.liang@postgrad.manchester.ac.uk}. This research was support by IBM and EPSRC in the form of an Industrial Case Doctoral Studentship Award.}
	\and Estelle Massart\thanks{ICTEAM Institute, UCLouvain, Euler Building, Avenue Georges Lemaître, 4 - bte L4.05.01, Louvain-la-Neuve, B - 1348, Belgium, \texttt{estelle.massart@uclouvain.be}. This research was supported by the National Physical Laboratory (Teddington, UK) and by the F.R.S.-FNRS (Belgium). } 
	\and
	Adilet Otemissov\thanks{Department of Mathematics, School of Sciences and Humanities, Nazarbayev University, Kabanbay Batyr 53, Astana 010000, Kazakhstan; \texttt{aotemissov@nu.edu.kz}. This research has been funded by Nazarbayev University under Faculty-development competitive research grants program for 2023--2025 Grant \textnumero 20122022FD4138.}
}

\date{\today}
\maketitle

{\small \begin{abstract}
 We propose an algorithmic framework, that employs active subspace techniques, for scalable global optimization of functions with low effective dimension (also referred to as low-rank functions). This proposal replaces the original high-dimensional problem by one or several lower-dimensional reduced subproblem(s), capturing the main directions of variation of the objective which are estimated here as the principal components of a collection of sampled gradients. We quantify the sampling complexity of estimating the subspace of variation of the objective in terms of its  effective dimension  and hence, bound the probability that the reduced problem will provide a solution to the original problem. To account for the practical case when the effective dimension is not known a priori, our framework  adaptively solves a succession of reduced problems,  increasing the number of sampled gradients until the estimated subspace of variation remains unchanged. We prove global convergence under mild assumptions on the objective, the sampling distribution and the subproblem solver, and illustrate numerically the benefits of our proposed algorithms over those using random embeddings. 
\end{abstract}

	\bigskip
	
	\begin{center}
		\textbf{Keywords:}
		global optimization, dimensionality reduction techniques, active subspaces, functions with low effective dimension,  low-rank functions. 
	\end{center}
}

\section{Introduction} \label{sec:intro}

We consider the unconstrained global optimization problem
\begin{align}\tag{P}
\label{P}
	f^{*}=\min_{\bx \in \bR^{D}} f(\bx)
\end{align}
where $f: \bR^{D} \rightarrow \bR$ is a real-valued continuously differentiable, possibly non-convex, deterministic function. We further assume $f$ to have \emph{low effective dimension}, namely, $f$ varies only within a (low-dimensional and unknown) linear subspace and is constant along its complement. These functions, also referred to as \emph{multi-ridge} \cite{Tyagi2014}, or \emph{low-rank} functions \cite{Cosson2023,Parkinson2023}, arise when tuning (over)parametrized models and processes, such as in hyper-parameter optimization for neural networks \cite{Bergstra2012}, heuristic algorithms for combinatorial optimization problems \cite{Hutter2014}, complex engineering and physical simulation problems \cite{constantine2015book} including climate modelling \cite{Knight2007}, and policy search and control \cite{Zhang19, Frohlich2020}. 
In \cite{cartis2023bound} and references therein, it was shown that the global optimization of these special-structure functions is tractable and scalable by means of random embeddings, namely, by (possibly repeatedly) solving a projected problem in (small-dimensional) random subspaces. Here,
we explore the use of problem-based embeddings for the global optimization of low-rank objectives; in particular, we aim to learn the effective subspace of the function before or during the optimization process.  


Our proposed framework builds on \emph{active subspaces}, a key notion for dimensionality reduction of parameter studies in physical/engineering models \cite{constantine2015book}. The active subspace of a function $f : \R^D \to \R$ is defined as the leading eigenspace of the matrix 
\begin{equation} \label{C}
    \bC = \int \nabla f(\bx) \nabla f(\bx)^T \rho(\bx) d\bx \in \bR^{D \times D},
\end{equation} 
where $\nabla f$ denotes the gradient of $f$ and $\rho$ is a probability density on $\R^D$.
Active subspaces aim to capture the main directions of  variation of the function $f$. In practice, the integral is approximated using Monte-Carlo sampling: 
\begin{equation} \label{eq:hatC} 
\hat{\bC} = \frac{1}{M}\sum_{j = 1}^M \nabla f(\bx^j_S) \nabla f(\bx^j_S)^T,
\end{equation}
where $\bx_S^1, \dots, \bx_S^M \in \R^{D}$ are samples drawn independently at random, according to the density $\rho$. The active subspace is then approximated by the leading eigenspace of $\hat{\bC}$, and its dimension is usually selected based on the decrease of the eigenvalues of $\hat \bC$ in order to achieve a good trade-off between complexity of the reduced model and discarded information. 

In this work, we exploit active subspace techniques for the global optimization of functions with low effective dimension. In the case where $f$ has effective dimension $d_e$ and $\rho$ is supported in the whole $\R^D$, we show that the matrix $\bC$ has rank $d_e$, and its range (or, equivalently, the $d_e$-dimensional active subspace of $f$) is equal to the effective subspace of variation of the objective. The original problem \eqref{P} can thus be solved by embedding the objective into the range of $\bC$, which can be estimated by the range of $\hat \bC$ provided $M$ is  sufficiently large. The important question is then: 
\begin{center}
    \emph{How large does  $M$ need to be in order for the leading eigenspace of $\hat \bC$ to capture all directions of variation of the objective?}
\end{center} 
We provide a novel sample complexity bound to answer this question that exploits the low effective dimensionality of $f$, and thus improves upon existing  sampling complexity results.

Since sampling bounds depend on the effective dimension $d_e$, using them directly within an algorithm may be impractical as $d_e$ is often unknown a priori, in real-life applications. 
In such cases, we propose the following algorithmic framework. We  replace the original high-dimensional problem \eqref{P} by a sequence/collection of reduced/embedded problems
\begin{align}\tag{RP}
\label{RP}
	\min_{\by \in \bR^{d}} f(\bA\by + \bp),	
\end{align}
where $\bA$ is a $D \times d$ real matrix whose range approximates the effective subspace of variation of $f$, $\bp \in \R^D$ is typically chosen as an estimate of a solution found so far in the run of the algorithm, and where we allow any global optimization solver to be used to solve the reduced problem \eqref{RP}. Successive embeddings (i.e., matrices $\bA$) are computed as bases of the range of $\hat \bC$, increasing progressively the number of samples $M$ in \eqref{eq:hatC}. For $M$ large enough, we show that the reduced problem \eqref{RP} is successful in the following sense.

\begin{definition}\label{def:RPsuccessful}
	We say that~(\ref{RP}) is successful if there exists $\by \in \bR^{d}$ such that $f(\bA\by + \bp) = f^{*}$. 
\end{definition}
In other words, \eqref{RP} is successful if it is equivalent to the original problem \eqref{P} in terms of optimal objective value; in that case, any optimal solution $\by^*$ of \eqref{RP} provides us with an optimal solution $\bx^* := \bA \by^* + \bp$ of the initial high-dimensional problem \eqref{P}. We prove the global convergence of our proposed algorithmic framework and quantify its complexity based on sampling complexity results for active subspaces.

\subsection{Related work.} 
Active subspaces capture the leading directions of variation of an arbitrary continuously differentiable function $f : \R^D \to \R$ by relying on the second moment matrix of the gradients \eqref{C}, see \cite{Samarov1993, Russi2010}. They are used in areas as diverse as dimensionality reduction in engineering/physical models \cite{constantine2015book,Lam2020}, Markov Chain Monte Carlo acceleration in Bayesian inference \cite{constantine2016b}, and neural networks' compression and vulnerability analysis \cite{Cui2020}; we refer to  \cite{constantine2015book} for a detailed exposition of active subspace theory including sampling complexity. Since active subspaces are estimated as the principal components of a set of sampled gradients, their sampling complexity is strongly connected to that of (streaming) PCA, see \cite{Huang2021} and the references therein.  Alternative methods for estimating the directions of variation of an arbitrary function include, for example, global sensitivity analysis \cite{Saltelli2008}, sliced inverse regression \cite{Garnett2014}, and basis adaptation \cite{Tipireddy2014}. Active subspaces (or related linear dimensionality reduction techniques) were used to improve the scalability of  genetic algorithms \cite{Demo2020}, shape optimization \cite{lukaczyk2014active},  and derivative-free methods \cite{Vandenbulcke2020}.


In parallel to the active subspace literature, several works specifically address the approximation of special-structure functions, including low-rank functions (equivalently, functions with low-effective dimension) \cite{Fornasier2012, Tyagi2014}. The authors of \cite{Fornasier2012} approximate low-rank functions over the unit ball when the gradients of the function are unknown and approximated using finite differences. They formulate the computation of the effective subspace of variation as a collection of compressed sensing problems involving finite differences in random directions and derive sampling complexity results. A variant was proposed in \cite{Tyagi2014}, that relies instead on a low-rank factorization formulation. Sampling complexity results for active subspace estimation are derived both in the latter works and in the dedicated literature using matrix concentration inequalities. Estimating the subspace of variation was shown to improve Bayesian optimization scalability for low-rank functions (see \cite{Binois2022} for a survey): the authors of \cite{Djolonga2013} use the compressed sensing approach in \cite{Tyagi2014} to learn the subspace of variation first and then optimize the function in this subspace, while \cite{Zhang19, Chen2020} rely on sliced inverse regression and update the estimated subspace during optimization as new information on the objective becomes available. The recent work \cite{Cosson2023} addresses first-order methods for local optimization of low-rank functions, by taking projected gradient steps within estimated active subspaces and deriving 
 improved global rates for local optimization.

Instead of approximating the effective subspace of variation, it was also proposed to use random embeddings for Bayesian \cite{Djolonga2013,Binois2014,Wang2016}, derivative-free   \cite{QianHuYu2016}, multi-objective \cite{Qian2017}, evolutionary \cite{Sanyang2016}, and global optimization of low-rank functions \cite{Binois2017,cartisOtemissov2022,cartis2023bound,cartis2023generalf}, as well as their local optimization by means of first-order \cite{becker, villa1, villa2} and second-order methods \cite{ZhenShao}. In \cite{cartisOtemissov2022}, we propose the REGO ({Random Embeddings for Global Optimization}) algorithmic framework, that replaces \eqref{P} by (possibly several realizations of) the reduced problem \eqref{RP} in which $\bA$ is a $D \times d$ Gaussian matrix. In the unconstrained case, we show that \eqref{RP} is successful with probability one if the random embedding dimension $d$ is at least the effective dimension of the objective $d_e$. In the bound-constrained case,  \cite{cartis2023bound} shows that multiple $d$-dimensional random embeddings (with $d \geq d_e$) are needed in order to find a global minimizer of \eqref{P}; however, when the effective subspace is aligned with coordinate axes, the number of required embeddings scales algebraically with the ambient dimension $D$. A convergence analysis of REGO for global optimization of Lipschitz-continuous objectives under no special-structure assumption, as well as for objectives with effective dimension $d_e > d$, is provided in \cite{cartis2023generalf}.  



\subsection{Contributions.}
We propose a generic global optimization framework for low-rank functions called \emph{Active Subspace Method for Global Optimization (ASM-GO)}, in which the original high-dimensional problem \eqref{P} is replaced by a collection of lower-dimensional subproblems \eqref{RP}, that can be solved by any global optimization solver. In contrast to REGO, ASM-GO relies on \emph{problem-based  embeddings}: the matrix $\bA$ in \eqref{RP} is a basis of the subspace spanned by a collection of sampled gradients $\{ \nabla f(\bx_S^1), \dots, \nabla f(\bx_S^{M})\}$, where $\bx_S^1, \dots, \bx_S^{M} \in \R^D$ are drawn independently at random according to some distribution $\rho$  supported\footnote{This is equivalent to selecting $\bA$ as a basis of the approximate active subspace of $f$ whose dimension is obtained by discarding zero eigenvalues.} in the whole $\R^D$. We emphasize that our algorithmic framework does not require any knowledge on the effective dimension/subspace of the objective. As the number of samples $M$ needed to generate a sufficiently accurate active subspace estimate is unknown, we progressively increase $M$ until reaching stagnation in the generation of $\bA$, which implies convergence of our algorithm. 
We show the following results.
\begin{itemize}
    \item The reduced problem \eqref{RP} is successful in the sense of \Cref{def:RPsuccessful} if $\bA$ is a basis of the effective subspace of variation of $f$; see \Cref{rp_as_suc}.
    \item The effective subspace of variation of $f$ is equal to the range of $\bC$, defined in \eqref{C}, which has dimension $d_e$ (also referred to as the active subspace of $f)$; see \Cref{thm:effective_equal_active}. 
    \item The range of $\hat \bC$, defined in \eqref{eq:hatC}, belongs to the effective subspace of variation of $f$. Its dimension is thus lower than or equal to the effective dimension $d_e$; see \Cref{lem:d<=d_e}. The reduced problem is thus successful if $\rank(\hat \bC) =d_e$; see \Cref{lem:d=de_implies_success}.
    \item Sampling complexity results allow us to lower bound the number of samples needed for the range of $ \hat \bC$ to have dimension $d_e$, hence, for the reduced problem to be successful; see \Cref{thm:estimate_on_M}. This bound on the probability of success involves the number of samples $M$, the spectrum of the second moment gradient matrix $\bC$, a uniform upper bound on the norm of the gradient, and the effective dimension $d_e$. In particular, note that there is no dependency in the sampling complexity on the initial dimension $D$. 
    \item Provided each reduced subproblem is solved to desired accuracy (with some probability), we show almost sure convergence of ASM-GO.
    \item We provide numerical comparisons of problem-based and random embedding frameworks applied to synthetic problems with induced low effective dimensionality, which show that practical ASM-GO variants perform competitively with REGO methods (the code is available on GitHub\footnote{See \url{https://github.com/aotemissov/Global_Optimization_with_Random_Embeddings}}).  
\end{itemize}

\subsection{Paper outline and notation}
\Cref{sec: prelim} introduces the low effective dimensionality assumption and relates effective and active subspaces. \Cref{sec: estim_success} quantifies the probability of success of \eqref{RP} in terms of the sampling number $M$, which is key for our global convergence analysis of \Cref{sec: algo}, where ASM-GO is also presented. Finally, \Cref{sec:numerics} contains numerical illustrations. In terms of notation, we use bold capital letters for matrices ($\mtx{A}$) and bold lowercase letters ($\mvec{a}$) for vectors, $\mtx{I}_D$ is the $D \times D$ identity matrix, and $\mvec{0}_D$, $\mvec{1}_D$ (or simply $\mvec{0}$, $\mvec{1}$) are the $D$-dimensional vectors of zeros and ones, respectively. We write $a_i$ to denote the $i$th entry of $\mvec{a}$ and write $\mvec{a}_{i:j}$, $i<j$, for the vector $(a_i \; a_{i+1} \cdots a_{j})^T$. We write $\langle \cdot , \cdot \rangle$, $\| \cdot \|$ (or equivalently $\| \cdot \|_2$) for the usual Euclidean inner product and Euclidean norm, respectively, and $\lceil \cdot \rceil$ for the ceiling operator.

\section{Functions with low effective dimensionality and active subspaces} \label{sec: prelim}
In this section we connect the two notions of the effective dimension of a function and its active subspace of variation.  
Let us first define functions with low effective dimension.

\begin{definition}[Definition 1 in Wang et al \cite{Wang2016}]
\label{def_f_low}
	A function $f: \bR^{D} \rightarrow \bR$ is said to have effective dimensionality $d_{e}$, with $d_{e} \leq D$, if
	\begin{itemize}
		\item there exists a linear subspace $\bT$ of dimension $d_e$ such that for all $\bx_{\top} \in \bT \subset \bR^{D}$ and $\bx_{\bot} \in \bT^{\bot} \subset \bR^{D}$, we have $f(\bx_{\top}+\bx_{\bot}) = f(\bx_{\top})$, where $\bT^{\bot}$ denotes the orthogonal complement of $\bT$;
	\item $d_{e}$ is the smallest integer with this property.
	\end{itemize}
	We call $\bT$ the effective subspace of $f$ and $\bT^{\bot}$ the constant subspace of $f$.
\end{definition}

\begin{example}[Extracted from \cite{cartisOtemissov2022}]
    The function $f(x_1,x_2) = \sin^2(x_1-x_2-0.5)$, depicted on \Cref{fig:low_dim}, has effective dimension $d_e = 1$, effective subspace $\spanvec([1,1])$ and constant subspace $\spanvec([1,-1])$.   
\begin{figure}[h!]
    \centering
  \includegraphics[scale=0.2, trim=0 50 0 70, clip]{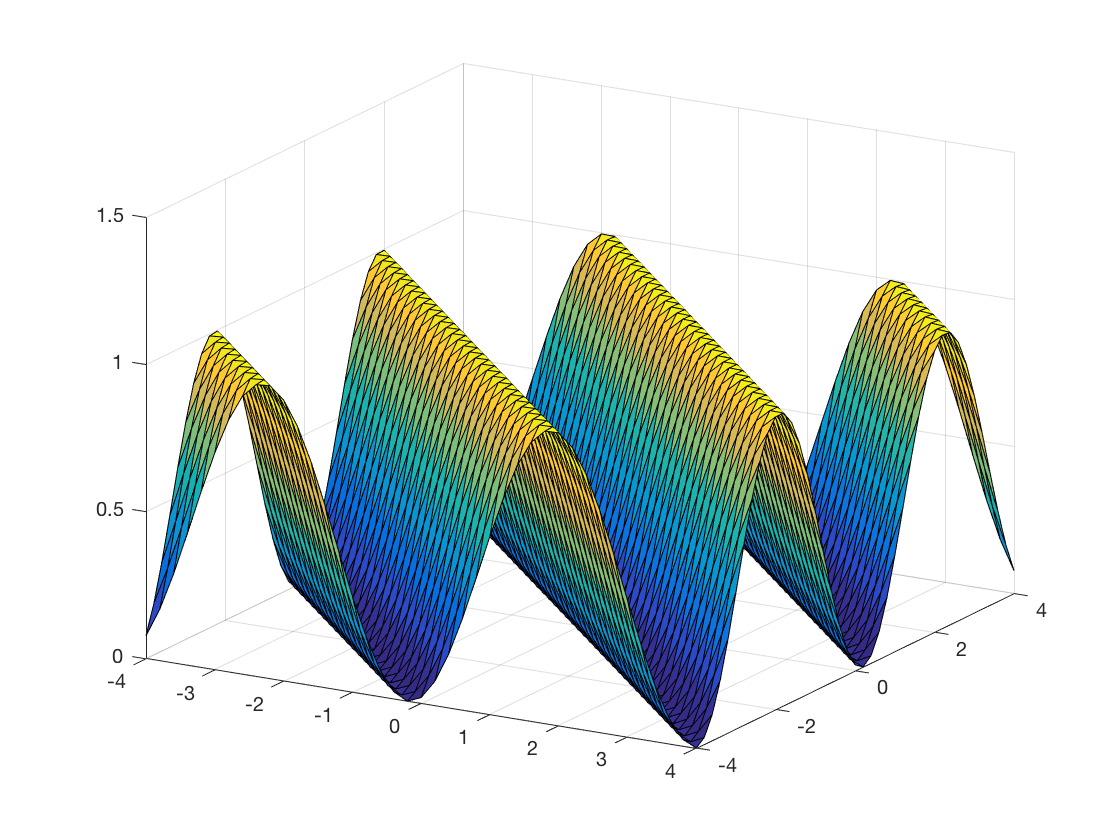}
    \caption{Example of function with low effective dimension  \cite{cartisOtemissov2022}.}
    \label{fig:low_dim}
\end{figure}
\end{example}

Let us now formalise our generic assumption on the structure of our objective function, which we require throughout the paper. We note that though we assume an effective subspace exists for $f$, 
we do not assume to know its orientation or a basis, a priori.
\begin{assumption} (Assumption 7.2 in  \cite{cartis2023generalf})
\label{ass:f_low_base}
	The function $f:\bR^{D} \rightarrow \bR$ has effective dimensionality $d_{e}$ with effective subspace $\bT$ and constant subspace $\bT^{\bot}$ spanned by the columns of  orthonormal matrices $\bU \in \bR^{D \times d_{e}}$ and $\bV \in \bR^{D \times (D-d_{e})}$, respectively. We write $\bx_{\top}=\bU\bU^{T}\bx$ and $\bx_{\bot}=\bV\bV^{T}\bx$, the unique Euclidean projections of any vector $\bx \in \bR^{D}$ onto $\bT$ and $\bT^{\bot}$, respectively.
\end{assumption} 
Provided \Cref{ass:f_low_base} holds, we have that for any $\bx \in \R^D$:
\begin{equation} \label{eq:eff_subspace}
     f(\bx) = f(\bU \bU^T \bx).
\end{equation}
Moreover, a matrix $\bU$ as in \Cref{ass:f_low_base} has the smallest possible dimensions, i.e., there is no matrix $\tilde \bU \in \R^{D \times n}$, $n < d_e$, such that $f(\bx) = f(\tilde \bU \tilde \bU^T \bx)$ for all $\bx \in \R^D$. This leads to the following definition.

\begin{definition} \label{def:g}
    Under \Cref{ass:f_low_base}, we define the low-dimensional representation $g : \R^{d_e} \to \R$ of the objective $f : \R^D \to \R$ as
    \[  f(\bx) =  f(\bU \bz) =: g(\bz),\]
with $\bz := \bU^T \bx$. 
\end{definition}

 We also rely on the following smoothness assumption on the objective. 
 
\begin{assumption} \label{ass:C1}
The function $f$ in \eqref{P} is continuously differentiable on $\R^D$.
\end{assumption}
The next results characterize the location of the gradients with respect to the effective and constant subspaces.

\begin{lemma} \label{lem:grad_in_T}
   Let \Cref{ass:f_low_base} and \Cref{ass:C1} hold. Then, $\nabla f(\bx) \in \calT$ for all $\bx \in \R^D$. 
\end{lemma}
\begin{proof}
       Let $g$ be defined in \Cref{def:g}. For all $\bx \in \R^D$, we have $\nabla f(\bx) = \bU \nabla g(\bz)$, so that $\nabla f(\bx) \in \range(\bU) = \calT$.
\end{proof}

\begin{lemma}\label{lem:grad_perp_Tperp}
Let \Cref{ass:f_low_base} and \Cref{ass:C1} hold. A vector $\bv \in \R^D$ belongs to the constant subspace $\calT^\perp$ if and only if  $\nabla f(\bx)^T \bv = 0$, for all $\bx \in \R^D$.   
\end{lemma}
\begin{proof}
    Note that $\bv \in \calT^\bot$ implies $\bv^\top \nabla f(\bx) = 0$ for all $\bx \in \R^D$, as $\nabla f(\bx) \in \calT$ by \Cref{lem:grad_in_T}. It remains to prove that $\nabla f(\bx)^T \bv = 0$ for all $\bx \in \R^D$ implies $\bv \in \calT^\perp$. The proof is by contradiction. Let us assume that $\bv \notin \calT^\perp$, so that $\tilde \bV := [\bV, \bv]$ has full column rank, with $\bV$ defined in \Cref{ass:f_low_base}. Due to \Cref{lem:grad_in_T} and  $\nabla f(\bx)^T \bv = 0$, we deduce that $\nabla f(\bx)^T \tilde \bV = 0$ for all $\bx \in \R^D$, i.e., the gradient of $f$ is everywhere orthogonal to the $(D-d_e+1)$-dimensional subspace spanned by the columns of $\tilde \bV$. Let $\tilde \bv_i$ be the $i$th column of $\tilde \bV$. Then, for all $\bx \in \R^D$ and for all $i = 1, \dots, D-d_e+1$, the fundamental theorem of calculus gives:
    \[  f(\bx+ t \tilde \bv_i) - f(\bx) =  \int_0^t \nabla f(\bx+\tau \tilde \bv_i)^T\tilde \bv_i \mathrm{d} \tau  = 0,\]
    so that for all $\bx \in \R^D$, the objective $f$ is constant on the one-dimensional subspaces $\bx + \range(\tilde \bv_i)$, $i = 1, \dots, D-d_e+1$. According to \Cref{def_f_low}, $f$ is of effective dimension at most $d_e-1$, which contradicts \Cref{ass:f_low_base}.
\end{proof}

The following lemma shows that if, ideally, we knew (a basis of) the effective subspace $\calT$, we could immediately reduce our original optimization problem (P) to solving the smaller-dimensional problem \eqref{RP} (in the subspace generated by this basis).
\begin{lemma}
\label{rp_as_suc}
	Let $f:\bR^{D} \rightarrow \bR$ satisfy Assumptions \ref{ass:f_low_base} and \ref{ass:C1}, and let $\bA = \bU$, with $\bU$ defined in\footnote{We note that $\bU$ can be any basis of the effective subspace $\calT$, and we refer to the notation in Assumption \ref{ass:f_low_base} for simplicity.} \Cref{ass:f_low_base}. Then, \eqref{RP} is successful according to \Cref{def:RPsuccessful}. 
\end{lemma}
\begin{proof}
    Let $\bx^*$ be a global minimizer of \eqref{P}, and let $\by^* := \bU^T (\bx^*-\bp)$. Then, $f(\bU \by^* + \bp) = f(\bU \bU^T (\bx^*-\bp) + \bp) = f(\bx_\top^* -\bp_\top + (\bp_\top + \bp_\perp)) = f^*$, with $\bp_\top = \bU \bU^T \bp$ and $\bp_\perp = \bV \bV^T \bp$, the unique projection of $\bp$ onto $\calT$ and $\calT^\perp$, respectively.
\end{proof}

The following result provides a characterization of the effective subspace $\calT$ as the range of a matrix obtained by aggregating gradient information; thus connecting to the active subspace notion mentioned in Introduction and opening the way to estimating $\calT$.

\begin{theorem}\label{thm:effective_equal_active}
    Let $f:\bR^{D} \rightarrow \bR$ satisfy \Cref{ass:C1}, and let $\rho$ have support over the whole $\R^D$. Let $\bC$ be defined as in \eqref{C}.
Then, $\rank(\bC) = d_e$ if and only if $f$ has effective dimensionality $d_e$. In this case, the effective subspace of $f$ is $\calT = \range(\bC)$.
\end{theorem}

\begin{proof}
    (Proof of If Part)
   Due to \Cref{lem:grad_in_T}, for all $\bx \in \R^D$, we know that $\nabla f(\bx) \in \calT$, so that $\nabla f(\bx) = \bU \bU^T \nabla f(\bx)$ and  
    \begin{equation}\label{eq:C=UC_redU^T}
        \bC = \bU \left( \int \bU^T \nabla f(\bx) \nabla f(\bx)^T \bU \rho(\bx) \mathrm{d}\bx \right) \bU^T.
    \end{equation}
    Denote $\bC_\red :=  \int  \bU^T \nabla f(\bx) \nabla f(\bx)^T \bU \rho(\bx) \mathrm{d}\bx \in \R^{d_e \times d_e}$ such that $\bC = \bU \bC_\red \bU^T$. 
    It is enough to show that $\bC_\red$ is full rank. By contradiction, let us assume that there exists $\bv \in \R^{d_e}, \bv \neq 0$ such that $\bC_\red \bv = 0$, hence, 
    \[\bv^T \bC_\red \bv = \int \bv^T \bU^T \nabla f(\bx) \nabla f(\bx)^T \bU \bv \rho(\bx) \mathrm{d}\bx = 0. \]
    Since $\rho(\bx) > 0$ for all $\bx \in \R^{D}$, we have $\bv^T \bU^T \nabla f(\bx) \nabla f(\bx)^T \bU \bv = (\nabla f(\bx)^T \bU \bv)^2 = 0$ for all $\bx \in \R^{D}$. By definition of $\bU$, the vector $\bU \bv$ belongs to the effective subspace $\calT$, while $\nabla f(\bx)^T \bU \bv = 0$ for all $\bx \in \R^D$ implies $\bU \bv \in \calT^\bot$ (see \Cref{lem:grad_perp_Tperp}). The only possibility is then $\bU \bv = \bzero$, which implies $\bv = \bzero$ as $\bU$ has full column rank. This contradicts our assumption that $\bv \neq \bzero$, and proves that $\bC_\red$ has full rank. 
    
     (Proof of Only If part)
    We need to prove that $\rank(\bC) = d_e$ implies i) that there is a linear subspace $\mathcal{T}$ of dimension $d_e$ such that
    \begin{equation}\label{eq:thm_f(x)=f(x_top)}
        f(\bx_{\top} + \bx_{\perp}) = f(\bx_{\top}) \;\; \text{for all $\bx_{\top} \in \mathcal{T}$ and $\bx_{\perp} \in \mathcal{T}^{\perp}$}
    \end{equation}
    and ii) that $d_e$ is the smallest integer for which part i) holds. 

    For part i), we claim that \eqref{eq:thm_f(x)=f(x_top)} works for $\mathcal{T} = \range(\bC)$. Let $\bW \in \R^{D \times d_e}$ be any orthonormal basis of $\range(\bC)$, and let $\bW^\perp \in \R^{D \times (D-d_e)}$ be an orthonormal basis of the orthogonal complement of $\range(\bC)$; then, note that the matrix $[\bW\, \,\bW^\perp] \in \R^{D \times D}$ is orthogonal.  For any $\bx \in \mathbb{R}^{D}$, we can write 
    \begin{equation*}
        \bx = \bW \bW^T \bx + \bW^\perp \bW^{\perp T} \bx = \bW \by + \bW^\perp \bz,
    \end{equation*}
    where $\by := \bW^T \bx$ and $\bz := \bW^{\perp T} \bx$. Let us now define $h(\bz) := f(\bW \by + \bW^\perp \bz)$. Note that then, $\nabla h(\bz) =\bW^{\perp T} \nabla f(\bx)$. 
    Furthermore,
    \begin{equation*} 
    \begin{split}
       & \int{\left( \nabla h(\bW^{\perp T} \bx) \right)^\top \left( \nabla h(\bW^{\perp T} \bx) \right) \rho(\bx) \,d\bx}
         \\&= \trace \left( \int{ \nabla h(\bW^{\perp T} \bx)  \nabla h(\bW^{\perp T} \bx)^T \rho(\bx) \,d\bx} \right) \\&
         = \trace \left( \bW^{\perp T} \left( \int{ \nabla f(\bx) \nabla f(\bx)^\top \rho(\bx) \,d\bx} \right) \bW^\perp \right) 
         = \trace \left( \bW^{\perp T} \bC \bW^\perp \right),
        \end{split}
    \end{equation*}
    which is zero by definition of $\bW^\perp$. Hence,
    \begin{equation*} 
        \int{||\nabla h(\bW^{\perp T} \bx)||_{2}^{2} \rho(\bx) \,d\bx} = \int{\left(\nabla h(\bW^{\perp T} \bx)\right)^T  \nabla h(\bW^{\perp T} \bx) \rho(\bx) \,d\bx} = 0.
    \end{equation*}
    Since $\rho(\bx) > 0$ everywhere in the domain, this implies that 
    \begin{equation} \label{eq:gradient_equals_matrix_transpose}
        \nabla h(\bW^{\perp T} \bx) = \bW^{\perp T} \nabla f(\bx) = 0 \;\; \text{for all $\bx \in \mathbb{R}^{D}$.}
    \end{equation}
    
   Now, let $\bx_{\top}$ be any vector in $\range(\bW)$ and $\bx_{\perp}$ be any vector in $\range(\bW^\perp)$. According to the mean value theorem, there exists $c \in [0,1]$ such that
   \begin{equation*}
       \begin{aligned}
        f(\bx_{\top} + \bx_{\perp}) & = f(\bx_{\top}) + \left( (\bx_{\top} + \bx_{\perp}) - \bx_{\top} \right)^T \nabla f \left( (1-c)\bx_{\top} + c(\bx_{\top} + \bx_{\perp}) \right)  \\
        & = f(\bx_{\top}) + \bx_{\perp}^T \nabla f(\bx_{\top} + c\bx_{\perp})  \\
        & = f(\bx_{\top}) + \left( \bW^\perp \bz \right)^T \nabla f(\bx_{\top} + c\bx_{\perp}) \\
        & = f(\bx_{\top}),
    \end{aligned}
   \end{equation*}
    where we used $\bx_{\perp} = \bW^\perp \bz$ (since $\bx_{\perp} \in \range(\bW^\perp)$) for some $\bz$ and \eqref{eq:gradient_equals_matrix_transpose}. This proves part i).
    
    We prove part ii) by contradiction. Suppose that there exists an integer $k < d_e$ for which \eqref{eq:thm_f(x)=f(x_top)} holds. Then, according to \Cref{def:g} there exists a function $g : \mathbb{R}^k \rightarrow \mathbb{R}$ such that $f(\bx) = g(\bar{\bU}^T\bx)$ for some orthonormal matrix $\bar{\bU} \in \mathbb{R}^{D \times k}$;  implying,
    \begin{align} \label{eq:proof_by_contradiction_d_e}
        \bC = \int{(\nabla f(\bx))(\nabla f(\bx))^\top} \rho(\bx) \,d\bx & = \int{\bar{\bU} \nabla g(\bar{\bU}^\top \bx) \left( \nabla g(\bar{\bU}^\top \bx) \right)^\top \bar{\bU}^\top \rho(\bx) \,d\bx} \nonumber \\
        & = \bar{\bU} \left( \int{\nabla g(\bar{\bU}^\top \bx) \left( \nabla g(\bar{\bU}^\top \bx) \right)^\top \rho(\bx) \,d\bx} \right) \bar{\bU}^\top.
    \end{align}
    Note that the rank of matrix in \eqref{eq:proof_by_contradiction_d_e} is at most $k$, which contradicts the fact that $\rank(\bC) = d_e > k$.

\end{proof}

\section{Estimating the active subspace}
\label{sec: estim_success}

Although \Cref{rp_as_suc} guarantees that \eqref{RP} is successful for $\bA = \bU$, the computation of $\bU$ according to \Cref{thm:effective_equal_active} relies on $\bC$, which cannot be computed exactly and is typically estimated using Monte Carlo sampling, as recalled in \Cref{alg:est_C}.

\begin{algorithm}
\begin{algorithmic}[1]
\State Generate $M$ samples $\{ \bx_{S}^j \}$ independently from the probability density function $\rho$.
\State Compute 
$\nabla f(\bx_{S}^j)$ for $j=1,\cdots,M$.
\State Return the estimate $\hat \bC$:
 \begin{equation*}
    \hat{\bC} := \frac{1}{M} \sum_{j=1}^{M}\nabla f(\bx_{S}^j)\nabla f(\bx_{S}^j)^{T}.
\end{equation*}
\end{algorithmic}
\caption{Estimating the gradient matrix $\bC$ by random sampling, extracted from \cite[Algorithm 3.1]{constantine2015book}.}
\label{alg:est_C}
\end{algorithm}


The main goal of this section is to relate the probability of success of the reduced problem \eqref{RP} with $\bA = \hat \bU$, where $\hat \bU \in \R^{D \times d}$ is any full column rank matrix whose columns span $\range(\hat \bC)$, to the sampling number $M$. The next two lemmas show that if $\range(\hat \bC)$ has dimension $d_e$, then \eqref{RP} is successful for this choice of $\bA$.  

\begin{lemma}\label{lem:d<=d_e} Let $f$ satisfy Assumptions \ref{ass:f_low_base} and \ref{ass:C1}, $\rho$ be a probability distribution on $\R^D$, $\hat \bC$ be defined according to \eqref{eq:hatC}, and $d := \rank(\hat \bC)$. Then, $d \leq  d_e$. 
\end{lemma}

\begin{proof}
By \Cref{def:g} there holds
$$ \hat \bC = \frac{1}{M} \sum_{j=1}^{M}\nabla f(\bx_{S}^j)\nabla f(\bx_{S}^j)^{T} = \bU  \left(\frac{1}{M} \sum_{j=1}^{M}\nabla g( \bU^T \bx_{S}^j)\nabla g( \bU^T \bx_{S}^j)^{T} \right) \bU^T. $$
Let us define $\hat \bC_\red := \frac{1}{M} \sum_{j=1}^{M}\nabla g( \bU^T \bx_{S}^j)\nabla g( \bU^T \bx_{S}^j)^{T}$ so that $\hat \bC = \bU \hat \bC_\red \bU^T$. As $\hat \bC_\red \in \R^{d_e \times d_e}$, $\rank(\hat \bC) \leq d_e$.  
\end{proof}

\begin{lemma} \label{lem:d=de_implies_success}
Let $f$ satisfy Assumptions \ref{ass:f_low_base} and \ref{ass:C1}, $\rho$ be a probability distribution on $\R^D$, $\bC$ and $\hat \bC$ be defined according to \eqref{C} and \eqref{eq:hatC} respectively, and $\hat \bU \in \R^{D \times d}$ be a full column rank matrix whose columns span $\range(\hat \bC)$. The reduced problem \eqref{RP} is successful with $\bA = \hat \bU$ if $d = d_e$.
\end{lemma}
\begin{proof} 
\eqref{RP} is successful with $\bA = \hat \bU$ if there exists $\by \in \R^{d_e}$ such that $f(\hat \bU \by + \bp) = f^*$. In \Cref{rp_as_suc}, we showed that, for any global minimizer $\bx^*$ of \eqref{P}, $\by^* = \bU^T (\bx^*-\bp)$ satisfies $f(\bU \by^* + \bp) = f^*$. It is thus sufficient to prove that $d = d_e$ implies $\range(\hat \bU) = \range(\hat \bC) = \range(\bU)$. The first equality holds by definition. For the second, by the proof of \Cref{lem:d<=d_e}, $\hat \bC = \bU \hat \bC_\red \bU^T$. If $d = \rank(\hat \bC) =d_e$, then  $\hat \bC_\red$ must have full rank, and $\range(\hat \bC) = \range(\bU)$.  
\end{proof}

These two lemmas prove that \eqref{RP} is successful if the set $\{ \nabla f(\bx_S^j), \dots, \nabla f(\bx_S^M)\}$ contains $d_e$ linearly independent vectors; in this case, the span of these vectors contains all required information for solving the original problem \eqref{P}. Since $\hat \bC$ is a random matrix, the event that it has rank $d_e$ is a random event. 
If the $d_e$th largest eigenvalue of $\hat \bC$ is a continuous random variable,  it is nonzero with probability one, so that \eqref{RP} is successful with probability one. Unfortunately, we cannot guarantee continuity due to the weak assumptions on $\rho$ and $f$, as illustrated by the following example.

\begin{example} \label{ex:polynomial_example} Consider the function $\bar f : \R \to \R$ illustrated in \Cref{fig:polynomial_example} and defined as follows: 
\begin{equation}  \label{eq:polynomial_example}
\bar f(x) = \left\{ \begin{array}{lr} 
0 & x < -1 \\
-x^4+2x^2-1 & x \in [-1,1] \\
0 & x > 1 \\
\end{array}\right. .
\end{equation}
Let $\rho$ be the uniform distribution on $[-2,2]$, and $x_S^1, \dots, x_S^M$ be sampled independently at random from $\rho$. With probability $(1/2)^M$,  $x_S^i \in [-2,-1] \cup [1,2]$ for all $i$, hence, $\bar f'(x_S^i) = 0$ for all $i$ and $\hat \bC = 0$.

\begin{figure}[h!]
    \centering
\includegraphics[scale = 0.5]{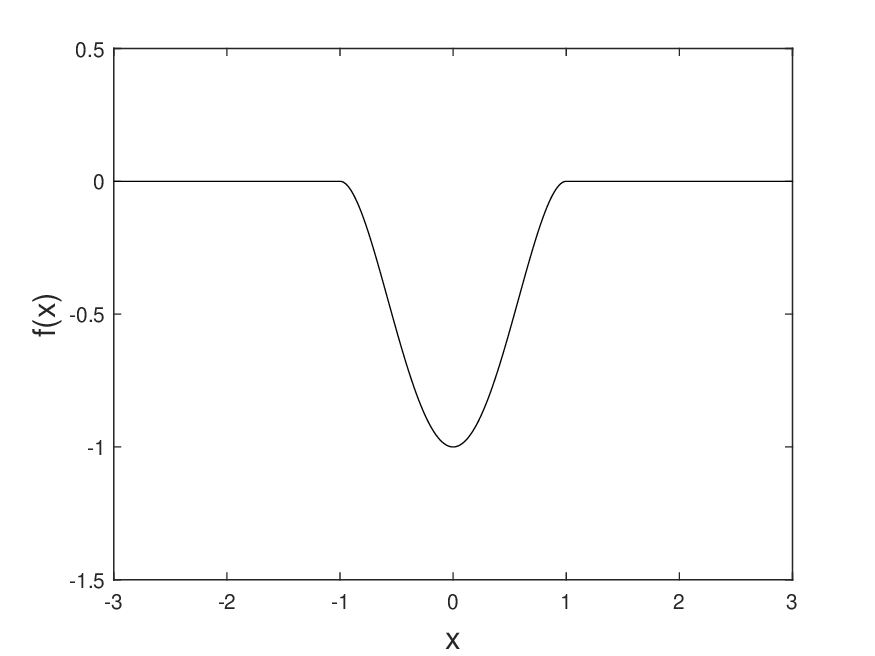}
    \caption{Illustration of the function $\bar f$ defined in \eqref{eq:polynomial_example}.}
    \label{fig:polynomial_example}
\end{figure}
\end{example}


\begin{remark} \label{rmk:computation_basis}The reduced problem \eqref{RP} involves $\hat \bU$, a basis of $\range(\hat \bC)$, whose  computation requires the sampled gradients to be sufficiently far from colinearity\footnote{In \Cref{sec:numerics}, we rely on the Gram-Schmidt procedure to compute the columns of $\hat \bU$. Though in theory orthonormality is not needed, arbitrary bases may lead to  bad conditioning for the reduced problem \eqref{RP}.}.
\end{remark}

The next result provides a lower bound on the number of samples $M$ needed for $\hat \lambda_{d_e}$, the $d_e$th largest eigenvalue of $\hat \bC$, to be at least $\tau \lambda_{d_e}$, with $\lambda_{d_e}$ the $d_e$th largest eigenvalue of the exact gradient matrix $\bC$ and $\tau \in [0,1)$ an arbitrary constant; taking $\tau = 0$ amounts to require $\hat \lambda_{d_e} > 0$. These bounds simply extend \cite[Section 7]{gittens2011} and its generalization to non-Gaussian distributions in \cite{constantine2015book} when applied to low-rank functions. The key underlying result is the matrix Bernstein inequality for sub-exponential matrices \cite[Theorem 5.3]{gittens2011}. 
To use this result, the matrices $\nabla f(\bx_S^j) \nabla f(\bx_S^j)^T$ must satisfy the subexponential moment growth condition (see \cite[Equation (3.26)]{constantine2015book}), which is guaranteed if the gradient of $f$ is uniformly bounded.  

\begin{assumption}\label{ass:BoundedGrad}
There exists $L > 0$ such that $\|\nabla f(\bx) \| \leq L$ for all $\bx \in \R^D$.
\end{assumption}

\begin{theorem}\label{thm:estimate_on_M}
    Let $f$ satisfy Assumptions \ref{ass:f_low_base}, \ref{ass:C1} and \ref{ass:BoundedGrad}, $\rho$ be a probability distribution with support the whole of $\R^D$, $\bC$ and $\hat \bC$ defined according to \eqref{C} and \eqref{eq:hatC} respectively, with eigenvalues\footnote{Note that, by \Cref{thm:effective_equal_active} and \Cref{lem:d<=d_e}, $\bC$ and $\hat \bC$ have at most $d_e$ nonzero eigenvalues.} $\lambda_1 \geq \lambda_2 \geq \dots \geq \lambda_{d_e} \geq 0$ and $\hat{\lambda}_1 \geq \hat{\lambda}_{2} \geq \dots \geq \hat{\lambda}_{d_e} \geq 0$. For all $\alpha \in (0,1]$, $\tau \in [0,1)$, and $k \in \{1, \dots, d_e\}$, 
    \begin{equation} \label{eq:bound_proba_tau}  \text{if} \qquad M \geq \frac{4 \lambda_1 L^2}{\lambda_{k}^2 (1-\tau)^2} \log\left(\frac{k}{\alpha}\right), \qquad \text{then} \qquad \prob\left[ \hat \lambda_{k} > \tau \lambda_{k} \right] \geq 1-\alpha.
    \end{equation}
    In particular, for all $\alpha \in (0,1]$,
    \begin{equation} \label{eq:final_bound_proba}
    \text{if} \qquad M \geq \frac{4 \lambda_1 L^2}{\lambda_{d_e}^2} \log\left(\frac{d_e}{\alpha}\right) \qquad \text{then}  \qquad \prob\left[ \eqref{RP} \; \text{is successful} \right] \geq 1-\alpha.
    \end{equation}
\end{theorem}

\begin{proof}
The first claim is an immediate consequence of \cite[Theorem 3.3]{constantine2015book} that shows
\begin{equation} \label{eq:bernstein}
    \prob \left[ \hat \lambda_{k}  \leq \tau \lambda_{k}  \right] \leq k \exp \left( \frac{-M\lambda_{k}^2 (1-\tau)^2}{4 \lambda_1 L^2} \right) \qquad \forall \tau \in [0,1), k = 1, \dots, d_e.
\end{equation}
For the second claim, \Cref{lem:d<=d_e} and \Cref{lem:d=de_implies_success} give
\begin{equation*}
    \prob\left[ \eqref{RP} \; \text{is successful} \right] = 1-\prob[d \neq d_e] = 1-\prob[d < d_e] = 1 - \prob[\hat \lambda_{d_e} = 0].
\end{equation*}
The claim then follows from letting $\tau = 0$ and $k = d_e$ in \eqref{eq:bernstein}.  
\end{proof}

Equation \eqref{eq:final_bound_proba} is a lower bound on the number of samples $\bx_S^1, \dots, \bx_S^M$ needed to ensure that the gradients $\nabla f(\bx_S^1), \dots, \nabla f(\bx_S^M)$ are linearly independent. On the other hand, \eqref{eq:bound_proba_tau} allows controlling the spectrum of $\hat \bC$, which is also important in light of \Cref{rmk:computation_basis}. 
The following example illustrates this bound on the function $\bar f : \R \to \R$ from \Cref{ex:polynomial_example}.
\begin{example}  \label{ex:back_to_polynomial_example_v2} Let $\rho$ be a standard Gaussian density, and let $\bar f$ be defined in \eqref{eq:polynomial_example}. It can be readily checked that $\bar f$ satisfies $|\bar{f}'(x)| \leq 8/(3 \sqrt{3})$ for all $x$, so that $L = 8/(3 \sqrt{3})$ in \Cref{ass:BoundedGrad}. 
On the other hand, we have 
\[  C = \int (\bar{f}'(x))^2 \rho(x) dx = 16 \int_{-1}^1 x^2 (1-x^2)^2 \rho(x) dx \simeq 0.83.\]
Equation \eqref{eq:final_bound_proba} with $d_e = 1$ becomes
\[ \text{if} \qquad M \geq \frac{4L^2}{C} \log(1/\alpha) \simeq 11.5 \log(1/\alpha) \qquad \text{then} \qquad \prob[\hat C \neq 0] \geq 1-\alpha.\] 
Note that the probability of $\hat C$ to be zero can be computed exactly. Let $x_S^1, \dots, x_S^M \in \R$ be sampled identically at random according to $\rho$, and let $\hat C := \frac{1}{M} \sum_{i = 1}^M (\bar{f}'(x_S^i))^2$. As $\bar{f}'(x)$ is zero if and only if $x \in ]-\infty, -1] \cup \{0\} \cup [1, \infty[$,
\[ \prob[\bar{f}'(x) = 0] = \int_{-\infty}^{-1} \rho(x) dx + \int_{1}^{\infty} \rho(x) dx = 1-(F_\rho(1)-F_{\rho}(-1)) \simeq 0.317,\]
where $F_\rho(x)$ is the cumulative density function of the standard Gaussian distribution. It follows that 
\[ \prob[\hat C \neq 0] = 1- (\prob[\bar{f}'(x) = 0])^M \simeq 1-0.317^M, \]
so that $M \geq 0.87 \log(1/\alpha)$ implies $\prob[\hat C \neq 0] \geq 1-\alpha.$ 
\end{example}

The discrepancy between the two lower bounds on $M$ in the previous example can be explained from the fact that the Bernstein inequality on which \eqref{eq:final_bound_proba} relies only exploits generic information on the density $\rho$, namely, $C$ and $L$, and does not use the flatness of $\bar f$ in some parts of its domain.

\begin{remark} \Cref{ex:back_to_polynomial_example_v2} shows that the curvature of the objective worsens the bound \eqref{eq:final_bound_proba} when $D=1$. Indeed, if $D = 1$, \eqref{eq:final_bound_proba} becomes 
\[ M \geq \frac{4L^2}{C} \log(1/\alpha), \]
where $C$ is the mean squared derivative and $L$ is the maximum derivative magnitude: 
\begin{equation*} 
    C = \int (f'(x))^2 \rho(x) dx \leq L^2, 
\end{equation*}
while for linear objectives\footnote{We mention here linear objectives to illustrate the theoretical properties of our bound, but highlight that these objectives are not of interest/considered in this paper.} $C = L^2$ so that \eqref{eq:final_bound_proba} simplifies to $M \geq 4 \log(1/\alpha)$. 
\end{remark}



Before concluding this section, we make two additional remarks on the scalability of the bounds \eqref{eq:bound_proba_tau} and \eqref{eq:final_bound_proba} with the dimension $d_e$, and their invariance under scaling of the objective and rotation of the coordinate axes. 

\begin{remark}\label{remark:M} The apparent logarithmic dependency of the lower bound \eqref{eq:final_bound_proba} with $d_e$ may seem surprising, as $\hat \bC$ has by construction a rank strictly lower than $d_e$ if $M < d_e$; we would thus expect instead a linear dependency between $M$ and $d_e$. This paradox can be explained by the dependency of the constant $L$ in the dimension $d_e$. Recall that, by \Cref{thm:effective_equal_active}, $\bC$ has rank $d_e$, and note that its trace satisfies 
\[ \sum_{i = 1}^{d_e} \lambda_i = \mathrm{tr} \left(\int \nabla f(\bx) \nabla f(\bx)^T \rho(\bx) d\bx \right) = \int \mathrm{tr} \left(\nabla f(\bx) \nabla f(\bx)^T \right) \rho(\bx) d\bx \leq L^2,  \]
where the last inequality follows from  \Cref{ass:BoundedGrad}. This implies $L^2 \geq d_e \lambda_{d_e}$, so that \eqref{eq:final_bound_proba} becomes\footnote{The case $\alpha = 1$ is trivial; it amounts to requiring $\prob[\eqref{RP} \; \text{is successful}] \geq 0$.}
\begin{equation} \label{eq:paradox}
     M \geq \frac{4 \lambda_1 d_e}{\lambda_{d_e}} \log \left(\frac{d_e}{\alpha} \right) \geq 4 d_e \log\left(\frac{d_e}{\alpha}\right) > 4 d_e \log(d_e)
\end{equation} 
for $\alpha \in (0,1)$. If $d_e = 1$, and since $M$ is by definition an integer, \eqref{eq:paradox} becomes $M \geq 1 = d_e$. For $d_e \geq 2$, \eqref{eq:paradox} gives $M \geq 4 \log(2) d_e \simeq 2.77 d_e \geq d_e$ as expected. 
\end{remark}


\begin{remark} Note that \eqref{eq:bound_proba_tau} and \eqref{eq:final_bound_proba} only depend on $f$ through the spectrum of the matrix $\bC$ and the upper bound $L$ on the gradient norm. In particular, these bounds are invariant under scaling of the objective and orthogonal transformation of the coordinates.
Indeed, for any $\beta \neq 0$ and any orthogonal matrix $\bQ \in \R^{D \times D}$, let $\bar f(\bx) := \beta f(\bQ \bx)$, let $\bar \bC := \int \nabla \bar f(\bx) \nabla \bar f(\bx) \rho(\bx) d\bx$, with eigenvalues\footnote{It is readily checked that, if $f$ has effective dimension $d_e$, then the same holds for $\bar f$, so that by \Cref{thm:effective_equal_active}, $\bar \bC$ has rank $d_e$. } $\bar \lambda_{1} \geq \dots \geq \bar \lambda_{d_e}$, and let $\bar L$ be an upper bound on the norm of the gradient of $\bar f$. As $\nabla \bar f(\bx) = \beta \bQ^T \nabla f(\bx)$, $\bar L = \beta L$ and the gradient matrix of $\bar f$ is $\bar \bC := \beta^2 \bQ^T \left( \int \nabla f(\bx) \nabla f(\bx) \rho(\bx) d\bx \right) \bQ$, which implies that $\bar \lambda_{i} = \beta^2 \lambda_i$ for all $i$. Lower bound \eqref{eq:final_bound_proba} becomes
\[ \bar M := \frac{4 \bar \lambda_{1} \bar L^2}{\bar{\lambda}_{d_e}^2} \log\left(\frac{d_e}{\alpha}\right) = \frac{4 \beta^4 \lambda_1 L^2}{\beta^4 \lambda_{d_e}^2 } \log\left(\frac{d_e}{\alpha}\right) = \frac{4  \lambda_1 L^2}{ \lambda_{d_e}^2} \log\left(\frac{d_e}{\alpha}\right). \]
The same reasoning shows the invariance of \eqref{eq:bound_proba_tau} under scaling of the objective and/or rotation of the coordinate axes. 
\end{remark}

\section{An active subspace method for global optimization} \label{sec: algo}
Let us now propose a generic algorithmic framework exploiting active subspaces for solving the global optimization problem (P). Our proposed framework, referred to as the Active Subspace Method for Global Optimization (ASM-GO), is presented in \Cref{alg:ASM}.  

\begin{algorithm}[H]
\begin{algorithmic}[1]
\State Initialize $\bp^{0} \in \bR^{D}$ 
\For{$k \geq 1$ until the termination criterion is satisfied}
\State Call Algorithm \ref{alg:est_C} with $M:=k$ to compute $\hat{\bC}$. 
\State Set $\bA^{k} = \hat{\bU}$, where $\hat{\bU}$ is any $D\times d_k$ matrix whose columns span $\range(\hat \bC)$.
\State Calculate $\by^{k}$ by solving approximately and possibly, probabilistically,
\begin{equation} \label{RP_k}
\tag{$\text{RP}^k$}\hspace*{-0.3cm}
    f_{\min}^{k}=\min_{\by \in \bR^{d_k}} f(\bA^{k} \by + \bp^{k-1}) 
\end{equation}
\State Construct $\bx^{k}=\bA^{k}\by^{k}+\bp^{k-1}$.
\State Choose $\bp^{k} \in \bR^{D}$ (deterministically or randomly) and let $k:=k+1$.
\EndFor
\end{algorithmic}
\caption{Active Subspace Method for Global Optimization (ASM-GO) of (P)}
\label{alg:ASM}
\end{algorithm}

At each major iteration $k$ of Algorithm \ref{alg:ASM}, a (new) sample of $k$ points is generated according to some distribution $\rho$, which is then used to estimate the active subspace using \Cref{alg:est_C}.  The reduced problem \eqref{RP} is solved in this ensuing subspace, and  the best point found is labeled as the next (major) iterate. Then a new sample of $k+1$ points is drawn, and the process is repeated until a suitable termination criterion is achieved - such as no significant change in the objective is obtained over major iterations, or the rank of $\hat\bC$ cannot be further increased; see Section 5 for more details on practical termination criteria. The parameter $\bp^k$ can be chosen for example as the best major iterate found so far, to ease the solution of \eqref{RP} and speed up that of (P); see also Section 5.



We next prove global convergence of ASM-GO based on the probability of success of the reduced problem derived in the previous section. Our proof follows a similar argument to the corresponding random embedding results in \cite{cartis2023bound}, but applied to a different context here, of active subspace learning; the intermediate results that are most similar to our approach in \cite{cartis2023bound} have been delegated to the appendix.

We first define the best point, $\bx_{\text{opt}}^{k}$, found up to the $k$th embedding as
\begin{align}
\label{x_opt}
	\bx_{\text{opt}}^{k} = \arg \min \{ f(\bx^{1}),f(\bx^{2}),...,f(\bx^{k}) \},
\end{align}
and we define $G_{\epsilon}$ the set of approximate global minimizers of \eqref{P}: 
\begin{equation}
\label{eq: G_epsilon}
G_{\epsilon} = \{ \bx \in \R^D : f(\bx) \leq f^* + \epsilon\}.
\end{equation}
 Note that, if for some $k$, the conditions a) and b) below are simultaneously satisfied, then $\bx^{k} \in G_{\epsilon}$:
\begin{enumerate}[label=\alph*)]
	\item The reduced problem~(\ref{RP_k}) is successful in the sense of \Cref{def:RPsuccessful}, namely,
	\begin{align}
	\label{e_l_f0}
	f_{\min}^{k} = f^{*}.
\end{align}
    \item The reduced problem~(\ref{RP_k}) is solved by a global solver to an accuracy $\epsilon$ in the objective function value, namely,
	\begin{align}
	\label{l_f0}
		f(\bA^{k}\by^{k}+\bp^{k-1}) \leq f_{\min}^{k} + \epsilon	
	\end{align}
	holds with a certain probability.
\end{enumerate}
We assume here that $\mvec{p}^k$ is a random variable; this covers deterministic choices of $\bp^k$ as a special case, by selecting a distribution supported on a singleton. 
Due to the stochastic definition of $\hat \bU$ (hence $\bA^{k}$) and of $\mvec{p}^k$, these conditions hold with a given probability. Let us therefore introduce two random variables that capture the conditions in (a) and (b) above,
\begin{align}
	R^k  &= \mathds{1}\{\text{\eqref{RP_k} is successful in the sense of \eqref{e_l_f0}}\}, \label{eq: Rk} \\ 
	S^k  &= \mathds{1}\{\text{\eqref{RP_k} is solved to accuracy $\epsilon$ in the sense of \eqref{l_f0}}\}, \label{eq: Sk}
\end{align}
where $\mathds{1}$ is the usual indicator function for an event. 

Let $\mathcal{F}^k = \sigma(\mtx{A}^1, \dots, \mtx{A}^k, \mvec{y}^1, \dots, \mvec{y}^k, \mvec{p}^0, \dots, \mvec{p}^k)$ be the $\sigma$-algebra generated by the random variables $\mtx{A}^1, \dots, \mtx{A}^k, \mvec{y}^1, \dots, \mvec{y}^k, \mvec{p}^0, \dots, \mvec{p}^k$ (a mathematical concept that represents the history of the algorithm as well as its randomness
until the $k$th embedding)\footnote{A similar setup regarding random iterates of probabilistic models can be found in \cite{Bandeira2014, Cartis2018} in the context of local optimization.}, with  $\mathcal{F}^0 = \sigma(\mvec{p}^0)$. 
We also construct an `intermediate' $\sigma$-algebra, namely,
$$\mathcal{F}^{k-1/2} = \sigma(\mtx{A}^1, \dots, \mtx{A}^{k-1}, \mtx{A}^{k}, \mvec{y}^1, \dots, \mvec{y}^{k-1}, \mvec{p}^0, \dots, \mvec{p}^{k-1}), $$
with  $\mathcal{F}^{1/2} = \sigma(\mvec{p}^0, \mtx{A}^{1})$.
Note that $\mvec{x}^k$, $R^k$ and $S^k$ are $\mathcal{F}^{k}$-measurable\footnote{It would be possible to restrict the definition of the $\sigma$-algebra $\mathcal{F}^k$
	so that it contains strictly the randomness of the embeddings $\mtx{A}^i$ and $\mvec{p}^i$ for $i\leq k$; then we would need to assume that $\mvec{y}^k$  
	is $\mathcal{F}^k$-measurable, 
	which would imply that $R^k$, $S^k$ and $\mvec{x}^k$ are also $\mathcal{F}^k$-measurable. Similar comments apply to the definition of 
	$\mathcal{F}^{k-1/2}$.}, and 
$R^k$ is also $\mathcal{F}^{k-1/2}$-measurable;
thus they are well-defined random variables.


\begin{remark}
    The random variables $\mtx{A}^1, \dots, \mtx{A}^k$, $\mvec{y}^1, \dots, \mvec{y}^k$, $\mvec{x}^1, \dots, \mvec{x}^k$, $\mvec{p}^0$, $\mvec{p}^1$, $\dots$, $\mvec{p}^k$, $R^1$, $\dots$, $R^k$, $S^1, \dots, S^{k}$ are  $\mathcal{F}^{k}$-measurable since $\mathcal{F}^0 \subseteq \mathcal{F}^1 \subseteq \cdots \subseteq \mathcal{F}^{k}$. Also, $\mtx{A}^1, \dots, \mtx{A}^k$, $\mvec{y}^1, \dots, \mvec{y}^{k-1}$, $\mvec{x}^1, \dots, \mvec{x}^{k-1}$, $\mvec{p}^0, \mvec{p}^1, \dots, \mvec{p}^{k-1}$, $R^1$, $\dots$, $R^k$, $S^1, \dots, S^{k-1}$ are  $\mathcal{F}^{k-1/2}$-measurable since $\mathcal{F}^0 \subseteq \mathcal{F}^{1/2} \subseteq \mathcal{F}^1 \subseteq \cdots \subseteq \mathcal{F}^{k-1} \subseteq \mathcal{F}^{k-1/2}$.
\end{remark}

A weak assumption is given next, that is satisfied by reasonable techniques for the subproblems; namely, the  reduced problem \eqref{RP_k} needs to be solved to required accuracy with some positive probability.

\begin{assumption}\label{assump: prob_of_I_R>rho}
	There exists $\gamma \in (0,1]$ such that, for all $k \geq 1$,\footnote{The equality in the displayed equation follows from $\mathbb{E}[S^k | \mathcal{F}^{k-1}] = 1 \cdot \prob[ S^k = 1 | \mathcal{F}^{k-1}] + 0 \cdot \prob[ S^k = 0 | \mathcal{F}^{k-1} ]$.}
\begin{equation*}
	  \prob[ S^k = 1 | \mathcal{F}^{k-1/2}]=\mathbb{E}[S^k | \mathcal{F}^{k-1/2}]  \geq \gamma, 
 \end{equation*}
	i.e., with (conditional) probability at least $\gamma > 0$, the solution $\mvec{y}^k$ of \eqref{RP_k} 
	satisfies \eqref{l_f0}.
\end{assumption}

\begin{remark}
 If a deterministic (global optimization) algorithm is used to solve \eqref{RP_k}, then $S^k$ is  always $\mathcal{F}_k^{k-1/2}$-measurable and \Cref{assump: prob_of_I_R>rho} is equivalent to $S^k\geq \gamma$. Since $S^k$ is an indicator function, this further implies that $S^k\equiv 1$, provided a sufficiently large computational budget is available.
\end{remark}


\begin{corollary} \label{corr:success_solver}
Let Assumption  \ref{assump: prob_of_I_R>rho} 
hold. Then, 
\begin{equation*}\label{ineq: exp_IS_IR>tau*rho}
\mathbb{E}[R^k S^k | \mathcal{F}^{k-1/2}]  \geq  \gamma R^k, \quad {\rm for}\quad k\geq 1.
\end{equation*}
\end{corollary}
\begin{proof}
Under \Cref{assump: prob_of_I_R>rho}, $\mathbb{E}[R^k S^k | \mathcal{F}^{k-1/2}] = R^k  \mathbb{E}[ S^k | \mathcal{F}^{k-1/2}] \geq  \gamma R^k$,
where the equality follows from the fact that $R^k$ is $\mathcal{F}^{k-1/2}$-measurable (see \cite[Theorem 4.1.14]{Durrett2019}).
\end{proof}

The results of \Cref{sec: estim_success} provide  a lower bound on the (conditional) probability of the reduced problem \eqref{RP_k} to be successful, which leads to the next corollary. 
\begin{corollary} \label{corr: lowerbdRPK}
Let Assumptions \ref{ass:f_low_base}, \ref{ass:C1} and \ref{ass:BoundedGrad} hold, let $\rho$ be a density with support in the whole of $\R^D$, and let $\lambda_1$ and $\lambda_{d_e}$ be defined in \Cref{thm:effective_equal_active}. If $k \geq M_0$, where
\begin{equation}\label{M_0}
    M_0 := \left\lceil \frac{4 \lambda_1 L^2}{\lambda_{d_e}^2} \log(d_e)  \right\rceil + 1
\end{equation}
then 
\begin{equation*}
\mathbb{E}[R^k | \mathcal{F}^{k-1}]  \geq \tau, \qquad \text{with} \qquad  \tau := 1- \exp\left( -\frac{\lambda_{d_e}^2}{4\lambda_1 L^2} \right).
\end{equation*}
\end{corollary}
\begin{proof}
Note that \begin{align*}
    k \geq \left\lceil \frac{4 \lambda_1 L^2}{\lambda_{d_e}^2} \log(d_e)  \right\rceil + 1  \geq \frac{4 \lambda_1 L^2}{\lambda_{d_e}^2} \log(d_e) + 1 = \frac{4 \lambda_1 L^2}{\lambda_{d_e}^2} \log\left( \frac{d_e}{ e^{-\lambda_{d_e}^2/(4\lambda_1L^2)} }\right).
\end{align*}
We then apply \Cref{thm:estimate_on_M} with $M=k$, $\alpha = e^{-\lambda_{d_e}^2/(4\lambda_1L^2)}$ to deduce that $\prob[R^k = 1 | \mathcal{F}^{k-1}] \geq 1 - e^{-\lambda_{d_e}^2/(4\lambda_1L^2)}$. Then, in terms of conditional expectation, we have
$\mathbb{E}[R^k | \mathcal{F}^{k-1}] = 1 \cdot \prob[ R^k = 1 | \mathcal{F}^{k-1}] + 0 \cdot \prob[ R^k = 0 | \mathcal{F}^{k-1} ] \geq \tau := 1 - e^{-\lambda_{d_e}^2/(4\lambda_1L^2)}$.
\end{proof}

We will use these two corollaries to prove global convergence of \Cref{alg:ASM}. Let us first state the following intermediary result, whose proof can be found in the Appendix.

\begin{lemma}\label{lemma: lim_of_prob_is_1}
Let Assumptions \ref{ass:f_low_base}, \ref{ass:C1}, \ref{ass:BoundedGrad} and \ref{assump: prob_of_I_R>rho} be satisfied, and let $\rho$ be a density with support in the whole of $\R^D$. Then, for $K\geq M_0$, where $M_0$ is defined in \eqref{M_0}, we have
	$$ \prob\Big[ \bigcup_{k=1}^K \left\{ \{R^k = 1\} \cap \{ S^k = 1 \} \right\} \Big] \geq 1 - (1- \tau \gamma)^{K-M_0 + 1}, $$
 where $\tau$ is defined in \Cref{corr: lowerbdRPK}.
\end{lemma}

In the next lemma, whose proof can be found in the Appendix, we show that if \eqref{RP_k} is successful and is solved to accuracy $\epsilon$ in objective value, then the solution $\mvec{x}^k$ must be inside $G_{\epsilon}$; thus proving our intuitive 
statements (a) and (b) at the start of this section. 

\begin{lemma}\label{lemma: if WcapG then_x in G_epsilon}
	Let Assumptions \ref{ass:f_low_base}, \ref{ass:C1}, \ref{ass:BoundedGrad} and \ref{assump: prob_of_I_R>rho} be satisfied, and let $\rho$ be a density  with support in the whole of $\R^D$. Then
		$$
		\{R^k = 1\}\cap \{S^k = 1\} \subseteq \{\mvec{x}^{k}  \in G_{\epsilon}\}.
		$$
\end{lemma}

We are now ready to prove the main result of this section.
\begin{theorem}[Global convergence]\label{thm: glconv}
	Let Assumptions \ref{ass:f_low_base}, \ref{ass:C1}, \ref{ass:BoundedGrad} and \ref{assump: prob_of_I_R>rho} be satisfied, and let $\rho$ be a density with support in the whole of $\R^D$ . Then
	$$\lim_{k\rightarrow \infty} \prob[\mvec{x}^k_{opt} \in G_{\epsilon}]=\lim_{k\rightarrow \infty} \prob[f(\mvec{x}^k_{opt}) \leq f^* + \epsilon] = 1$$
	where $\mvec{x}^k_{opt}$ and $G_{\epsilon}$ are defined in \eqref{x_opt} and   \eqref{eq: G_epsilon}, respectively.
	
	Furthermore, for any $\xi \in (0,1)$, 
	\begin{equation*}\label{eq:prob[x_opt^k_in_G_eps]>alpha}
\text{$\prob[ \mvec{x}^k_{opt} \in G_{\epsilon} ]= \prob[f(\mvec{x}^k_{opt}) \leq f^* + \epsilon]\geq \xi$ for all $k \geq K_{\xi}$,}
\end{equation*}
where $K_\xi:= \displaystyle\ceil*{\frac{|\log(1-\xi)|}{\tau \gamma}} + M_0 -1$ with $\tau$ and $M_0$ defined in \Cref{corr: lowerbdRPK}.
\end{theorem}
\begin{proof}
Lemma \ref{lemma: if WcapG then_x in G_epsilon} and the definition of $\mvec{x}^k_{opt}$ in \eqref{x_opt} provide 
	$$ \{ R^k = 1 \} \cap \{ S^k = 1 \} \subseteq \{ \mvec{x}^k \in G_{\epsilon} \} \subseteq \{ \mvec{x}_{opt}^k \in G_{\epsilon} \} $$
	for $k = 1, 2,\dots, K$ and for any integer $K\geq 1$. Hence, 
	\begin{equation}\label{rel: cup_I is in cup_X}
	\bigcup_{k=1}^K \{ R^k = 1 \} \cap \{ S^k = 1 \} \subseteq \bigcup_{k=1}^K \{ \mvec{x}^k_{opt} \in G_{\epsilon} \}.
	\end{equation}
	Note that  the sequence $\{ f(\mvec{x}^1_{opt}), f(\mvec{x}^2_{opt}), \dots, f(\mvec{x}^K_{opt})\}$ is monotonically decreasing. Therefore, if $\mvec{x}^k_{opt} \in G_{\epsilon}$ for some $k \leq K$ then $\mvec{x}^i_{opt} \in G_{\epsilon}$ for all $i = k, \dots, K$; and so the sequence $(\{ \mvec{x}^k_{opt} \in G_{\epsilon} \})_{k = 1}^K$ is an increasing sequence of events. Hence,
	\begin{equation}\label{eq:cup_x_opt_in_G=x_opt_in_G}
	    \bigcup_{k=1}^K \{ \mvec{x}^k_{opt} \in G_{\epsilon} \} = \{ \mvec{x}^K_{opt} \in G_{\epsilon} \}.
	\end{equation}
	From \eqref{eq:cup_x_opt_in_G=x_opt_in_G} and \eqref{rel: cup_I is in cup_X}, we have for all $K\geq 1$,
	\begin{equation}\label{eq:prob[x_opt_in_G_eps>1-(1-tr)^K]}
	\prob[\{ \mvec{x}^K_{opt} \in G_{\epsilon} \}]  \geq \prob\Big[\bigcup_{k=1}^K \{ R^k = 1 \} \cap \{ S^k = 1 \} \Big]  \geq 1 - (1- \tau \gamma)^{K-M_0+1},
	\end{equation}
		where the second inequality follows from \Cref{lemma: lim_of_prob_is_1}.
Passing to the limit in $K$ in \eqref{eq:prob[x_opt_in_G_eps>1-(1-tr)^K]}, we deduce
$1 \geq \lim_{K \rightarrow \infty} \prob[\{ \mvec{x}^K_{opt} \in G_{\epsilon} \}] \geq \lim_{K \rightarrow \infty} \left[1 - (1- \tau \gamma)^{K-M_0+1}\right] = 1$,
as required.
 Note that if 
    \begin{equation} \label{eq:1-(1-tau rho)^K>alpha}
        1-(1-\tau \gamma)^{k-M_0+1} \geq \xi
    \end{equation}
    then \eqref{eq:prob[x_opt_in_G_eps>1-(1-tr)^K]} implies $\prob[ \mvec{x}^k_{opt} \in G_{\epsilon} ] \geq \xi$. Since \eqref{eq:1-(1-tau rho)^K>alpha} is equivalent to
    $ k\geq \displaystyle\frac{\log(1-\xi)}{\log(1-\tau \gamma)}+M_0-1 $, 
     \eqref{eq:1-(1-tau rho)^K>alpha} holds for all $k\geq K_\xi$ since 
    $K_\xi \geq \displaystyle\frac{\log(1-\xi)}{\log(1-\tau \gamma)} + M_0 -1$.
\end{proof}

\section{Numerical experiments} \label{sec:numerics}

\subsection{Experimental setup}

We propose two variants of our ASM-GO  framework (Algorithm \ref{alg:ASM}) that we implement and test numerically (the code is available on Github\footnote{See \url{https://github.com/aotemissov/Global_Optimization_with_Random_Embeddings
}}). Algorithm \ref{A-ASM} (A-ASM) is a practical variant that aims to save on the number of gradient samples and basis (re-)computations.  Instead of completely re-sampling the gradient values that give the estimate matrix $\hat{\bC}$ and re-calculating a basis $\hat{\bU}$ at every iteration, as in  Algorithm \ref{alg:ASM}, A-ASM maintains a basis of the estimated active subspace, samples one new gradient per iteration and uses it to augment the basis by means of  Gram-Schmidt orthogonalization\footnote{Here, for simplicity, we use the classical form of this procedure; more sophisticated variants of (online) Gram-Schmidt orthogonalization could be considered.}. 

\begin{algorithm}[H]
\begin{algorithmic}[1]
\State Initialize $\bp^{0} \in \bR^{D}$. 
\State Sample a point $ \bx_S^{1} $  from the probability density function $\rho$, set 
 $\ba^{1}:=\nabla f(\bx_S^{1})/\|\nabla f(\bx_S^{1})\|_2$, 
  $\bA^{1}:=\ba^{1}$ and $d = 1$.
\For{$k \geq 1$ until the termination criterion is satisfied}
\If{$k > 1$}
\State Sample a new point $\bx_S^{k}$ from the probability density function $\rho$, and set $\ba^{k}:=\nabla f(\bx_S^{k})$.
\State Let $\ba^{k} = \ba^k - \text{proj}_{\ba^{1}}(\ba^k) - \cdots - \text{proj}_{\ba^{d}}(\ba^k),$
where $\text{proj}_{\bv}(\bu) = \frac{<\bv,\bu>}{<\bv,\bv>}\bv$.
\State If $\| \ba^{k}\|_2 \geq 10^{-6}$, set $\ba^{k} := \ba^{k}/\|\ba^{k}\|_2$, $\bA^{k} := [\bA^{k-1} \ \ba^{k}]$ and $d := d+1$.
\EndIf
\State Calculate $\by^{k}$ by solving approximately,
and possibly, with a given probability,
\begin{equation*}
f_{\min}^{k}=\min_{\by \in \bR^{d}} f(\bA^{k} \by + \bp^{k-1})
\end{equation*}
\State Construct $\bx^{k}=\bA^{k}\by^{k}+\bp^{k-1}$.
\State Choose $\bp^{k} \in \bR^{D}$ (deterministically or randomly) and let $k:=k+1$.
\EndFor
\end{algorithmic}
\caption{Adaptive-ASM (A-ASM) applied to~(\ref{P})}
\label{A-ASM}
\end{algorithm}

Algorithm \ref{ASM-1} (ASM-1)  generates only one  estimate of the active subspace of (P), from a user-chosen, hopefully sufficiently large, number $M$ of gradient samples, thus solving only one reduced subproblem. This variant is suitable when the effective dimension $d_e$ is known a priori, which then guides the choice of $M$ (following our theoretical developments in Section \ref{sec: estim_success}), saving on the number of subproblems that need solving.  We note that the basis $\bA$ in Algorithm \ref{ASM-1} is the same as  $\hat{\bU}$ in our earlier notation in say, Algorithm \ref{alg:ASM}.
\begin{algorithm}[H]
\begin{algorithmic}[1]
\State Initialize $\bp \in \bR^{D}$ and $M$, a number of sample gradients.
\State Apply Algorithm~\ref{alg:est_C} with $M$ samples to compute  $\hat{\bC}$.
\State Compute a basis $\bA$ of the range of $\hat{\bC}$, and let $d$ be its dimension.
\State Calculate $\by$   by solving approximately,
and possibly, with a given probability,
\begin{equation*}
\begin{aligned} 
	& \min_{\by \in \bR^{d}} f(\bA \by + \bp).
\end{aligned}
\end{equation*}
\State Construct the solution estimate, $\bx=\bA\by+\bp$.
\end{algorithmic}
\caption{ASM-1 applied to~(\ref{P})}
\label{ASM-1}
\end{algorithm}

\paragraph{Choice of global solver for the subproblems.} To solve the reduced subproblems in our ASM variants (and in other methods we test), we employ the KNITRO solver \cite{Byrd2006}, a package for solving large-scale nonlinear local optimization problems. To turn it into a global solver, we activate its multi-start feature (the solver is then referred to as mKNITRO). The solver has the following options/specifications:
(1) Computational cost measure: function evaluations, CPU time in seconds; (2)  Budget to solve a $d$-dimensional problem: $\min(200,10d)$ starting points; (3) Termination criterion: default options. (4) Additional options: \texttt{ms\_enable = 1}, \texttt{ms\_maxbndrange = 2}.

\paragraph{Summary of algorithms considered.}
We compare the above ASM variants, Algorithms \ref{A-ASM} and \ref{ASM-1}, with  the Random Embeddings for Global Optimization 
(REGO) framework \cite{cartisOtemissov2022, cartis2023generalf}, that at each iteration,  draws a random subspace (of suitable or increasing dimension) by means of a Gaussian matrix, and then globally optimizes the objective $f$ in this subspace using a global solver. 
To summarize, we consider the following variants of ASM and REGO:
\begin{itemize}
    \item \textbf{A-ASM:} In Algorithm~\ref{A-ASM}, the reduced problem is solved using mKNITRO and the point $\bp^{k}$ is chosen as the best point found so far, up to the $k$th embedding: $\bp^{k} =\bx^k= \bA^{k}\by^{k}+\bp^{k-1}$.
    \item \textbf{ASM-1:} This is Algorithm~\ref{ASM-1} with $M=d_e$ and where the reduced problem is solved using mKNITRO.
    \item \textbf{A-REGO}: This is Algorithm A-REGO in \cite{cartis2023generalf} (which we re-state in the Appendix as Algorithm~\ref{REGO}, for completeness). The reduced problem is solved using mKNITRO and the anchor point $\bp^{k}$ at which the random subspace is drawn is chosen as the best point found so far, up to the $k$th embedding: $\bp^{k} = \bx^k=\bA^{k}\by^{k}+\bp^{k-1}$.
    \item \textbf{REGO-1}: This is Algorithm REGO in \cite{cartisOtemissov2022} (which we re-state in the Appendix as Algorithm~\ref{REGO-1}, for completeness), with $d = d_e$ and where the reduced problem is solved with mKNITRO.
    \item \textbf{No-embedding:} The original problem \eqref{P} is solved using mKNITRO (with no subspace reduction or low-dimensional structure exploitation).
\end{itemize}

We terminate A-ASM and A-REGO when the following criterion is satisfied or after at most $D$ embeddings (i.e., when $d=D$ in \Cref{A-ASM} or $k = D$ in \Cref{REGO}).
\begin{enumerate}
\item For A-REGO, we stop the algorithm when stagnation is observed in the objective; namely, the algorithm stops after $k_{f}$ embeddings, where $k_{f}$ is the smallest $k \geq 2$ that satisfies
\begin{align}
	|f(\bA^{k}\by^{k}+\bp^{k-1})-f(\bA^{k-1}\by^{k-1}+\bp^{k-2})| \leq \gamma = 10^{-5}.
\end{align}

\item For A-ASM, we stop the algorithm if no new independent sample is generated over several successive attempts; namely, the algorithm stops after $k_g$ embeddings, where $k_g$ is the smallest $k \geq 2$ such that the norms of the vectors $\ba^{k_g-4}, \dots, \ba^{k_g}$ are less than $10^{-6}$. This means that $\bA^{k_g-4} = \bA^{k_g-3}  = \dots = \bA^{k_g}$, in other words, we have no improvement in estimating the effective subspace over the last 5 embeddings. 
\end{enumerate}
Note that A-ASM and A-REGO also return an estimate $d_{\text{est}}$ of the effective dimension of the problem. For A-REGO, if $k_{f} < D$, we let $d_{\text{est}} = k_{f}-1$; for A-ASM, if $k_{g} < D$, we let $d_{\text{est}} = k_{g}$; otherwise, $d_{\text{est}} = D$.\label{page:est_de}

\paragraph{Performance profiles.} We rely on the performance profiles introduced in \cite{dolan2002benchmarking} to present our results. For each function $f \in \func$ and algorithm $s \in \set$ (for test set $\func$ set of algorithms $\set$), we define
\begin{align}
	N_f(s) := \# \text{ of function evaluations required by algorithm $s$ to converge. }
\end{align}
If algorithm $s$ fails to converge to an  $\epsilon$-minimizer of $f$, for $\epsilon=10^{-3}$, within the allowed computational budget, we set $N_f(s) = \infty$. We define the performance probability as 
\begin{align}
	\pi_s(\alpha) = \frac{|\{ f \in \func: N_f(s) \leq \alpha \min_{s \in \set} N_f(s) \}|}{|\func|}, \qquad \forall \alpha \geq 1,
\end{align}
where $|\cdot|$ denotes the cardinality of a set. The same procedure can be applied when taking CPU time to be performance metric.

\paragraph{Test Set.} Our main test set is generated from a set of benchmark functions for global optimization modified to increase artificially the dimension of their domain to $D = 100$ or $D = 1000$ as in \cite{Wang2016} (see Appendix~\ref{appen: table} for more information on the test set generation). The resulting functions therefore have low effective dimension by construction.

\subsection{Numerical Results}


\paragraph{Experiment 1: Comparison of algorithms on the test set.} 
We compare ASM-1, A-ASM, REGO-1 and A-REGO with the no-embedding framework on the whole problem set, in terms of the function evaluation counts (see \Cref{fig:perf_nfuneval}) or CPU time (see \Cref{fig:perf_t}). To ensure a fair comparison between the different algorithms, we add a fixed number of function evaluations to the total count for A-ASM and ASM-1, to account for the computation of the sampled gradients. For ASM-1, the additional number of function evaluations is $(D+1)\times M$ (i.e., the cost to estimate $M$ gradients using finite differences); for A-ASM, one additional gradient is computed per embedding, leading to a total additional number of function evaluations equal to $(D+1) \times k_g$. When we compare performances in terms of CPU time, our results include the cost of generating $\bA^{k}$ in each algorithm (A-ASM, ASM-1, A-REGO, REGO-1).

Both the evaluation count profile in \Cref{fig:perf_nfuneval} and the CPU time performance in \Cref{fig:perf_t}
show that the single subspace methods (with a priori knowledge of $d_e$ and single subproblem solve) typically win over adaptive variants, as expected. In terms of adaptive variants, A-ASM is at least as good and sometimes better than A-REGO,
showing that when $d_e$ is unknown, it is worthwhile learning the effective subspace rather than performing random subspace selections.  

For ASM-1 and REGO-1, we assume that the effective dimension $d_e$ is known for all problems and select $M = d_e$ according to  the discussion in Section \ref{sec:choose_M}. By contrast, A-ASM and A-REGO do not require  any a priori knowledge of $d_e$; but can efficiently estimate it. Table \ref{table:estimate_d} shows the estimated effective dimension ($d_{\text{est}}$) returned by the adaptive algorithms (namely, A-ASM and A-REGO) at termination\footnote{This calculation is described at page \pageref{page:est_de}.}, for three random seeds/algorithm realizations, and compares it with the true effective dimension. When the estimated effective dimension differs across seeds, the estimates for the second and third seeds are presented inside brackets.
We note that in most cases, $d_{\rm est}$ coincides with $d_e$, and in 
all cases, at least one realization of each adaptive algorithm provides an exact estimate of $d_e$ and all estimates differ from $d_e$ by at most $1$.


\begin{figure}[h]
\subfloat{
  \centering
  \includegraphics[width=0.9\linewidth]{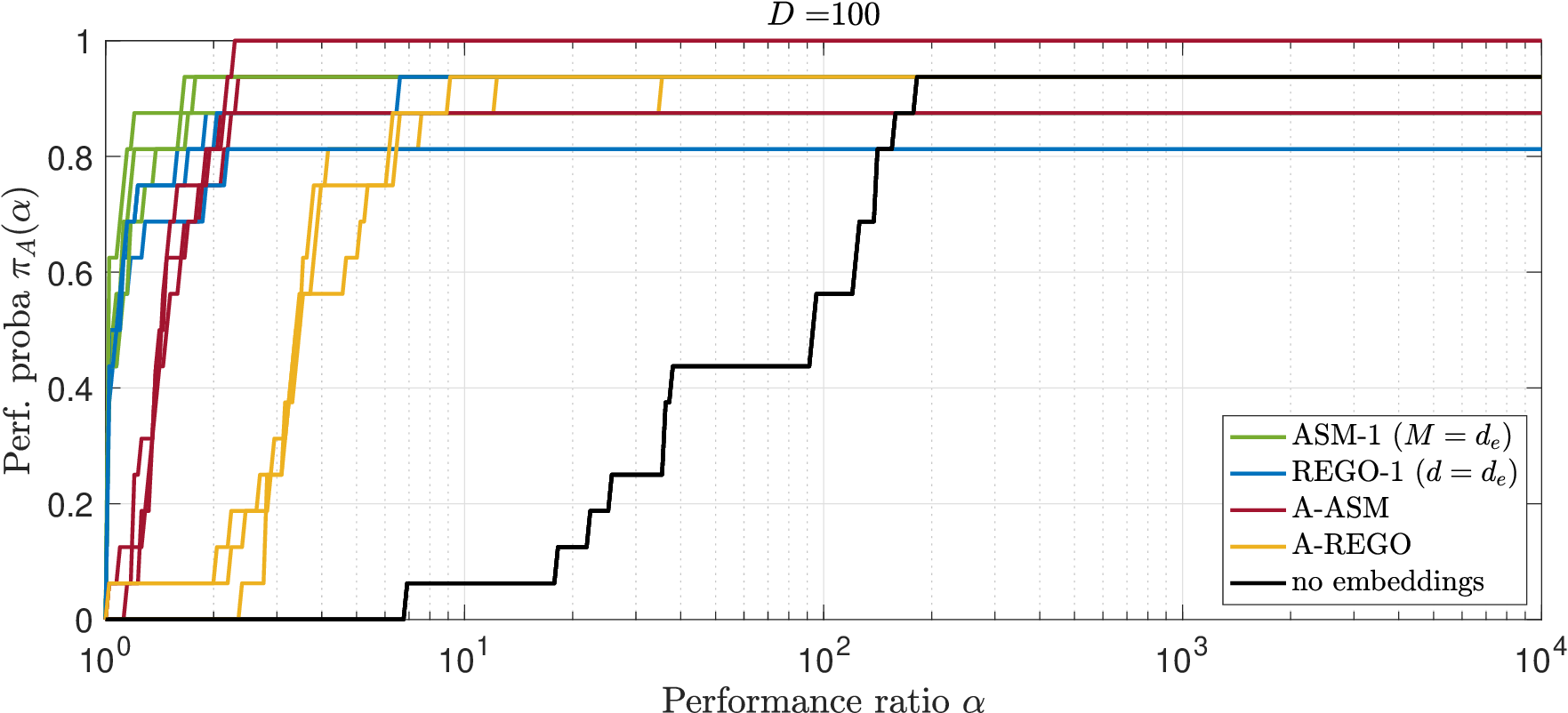} }
\newline
\subfloat{
  \centering
  \includegraphics[width=0.9\linewidth]{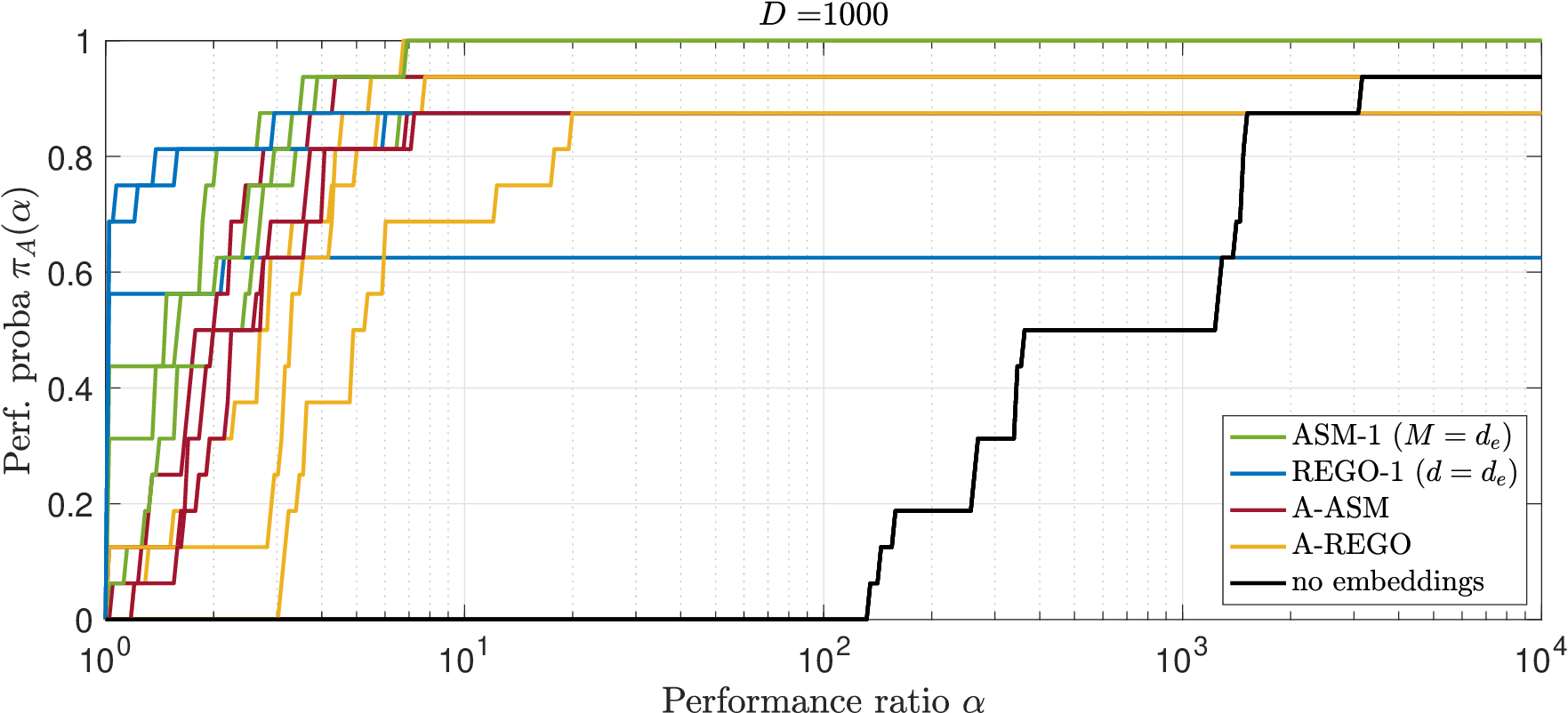} }
\caption{Comparing ASM-1, A-ASM, REGO-1, A-REGO and no-embedding (with mKNITRO) for functions with low effective dimensionality, in terms of function evaluation counts. Lines with the same colour represent three different realisations of an algorithm.}
\label{fig:perf_nfuneval}
\end{figure}

\begin{figure}[h]
\subfloat{
  \centering
  \includegraphics[width=0.9\linewidth]{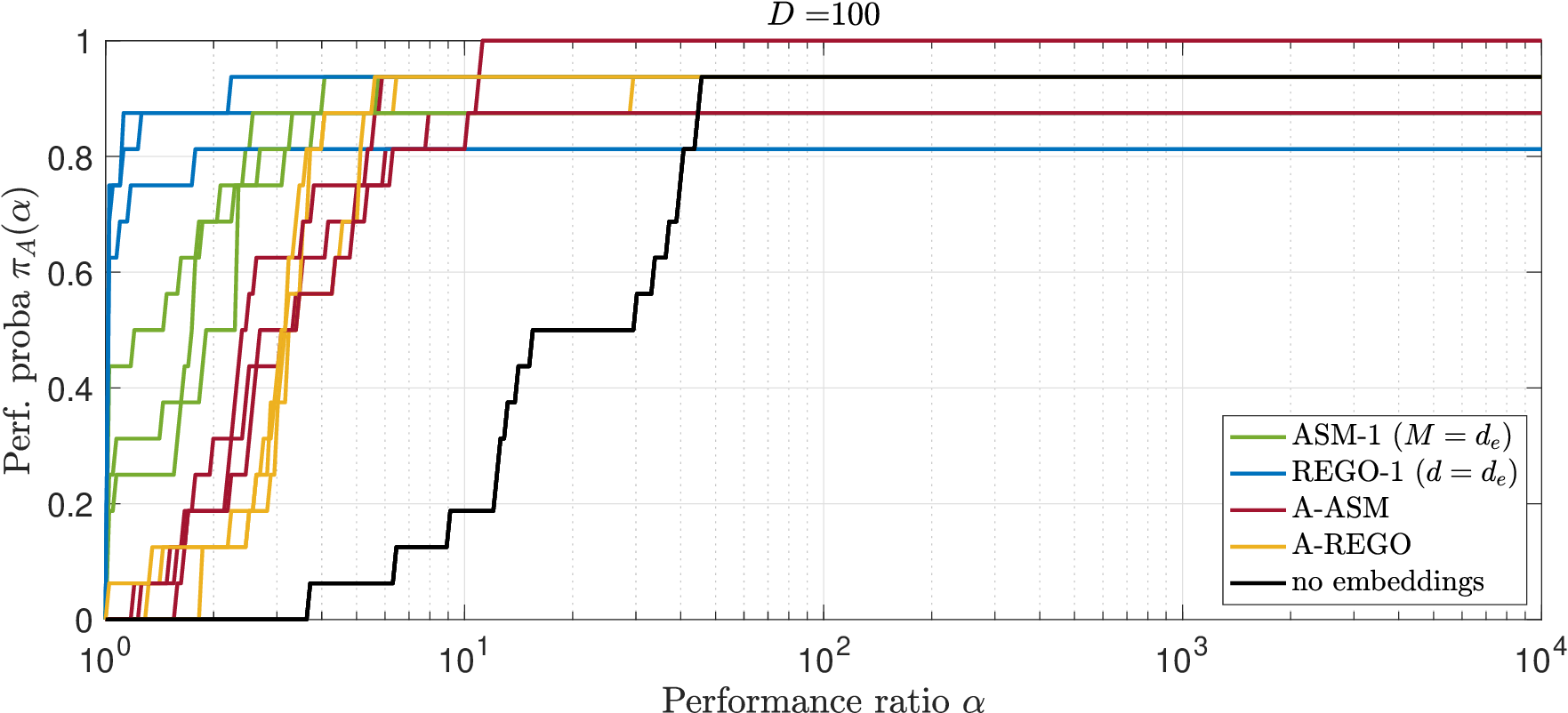} }
\newline
\subfloat{
  \centering
  \includegraphics[width=0.9\linewidth]{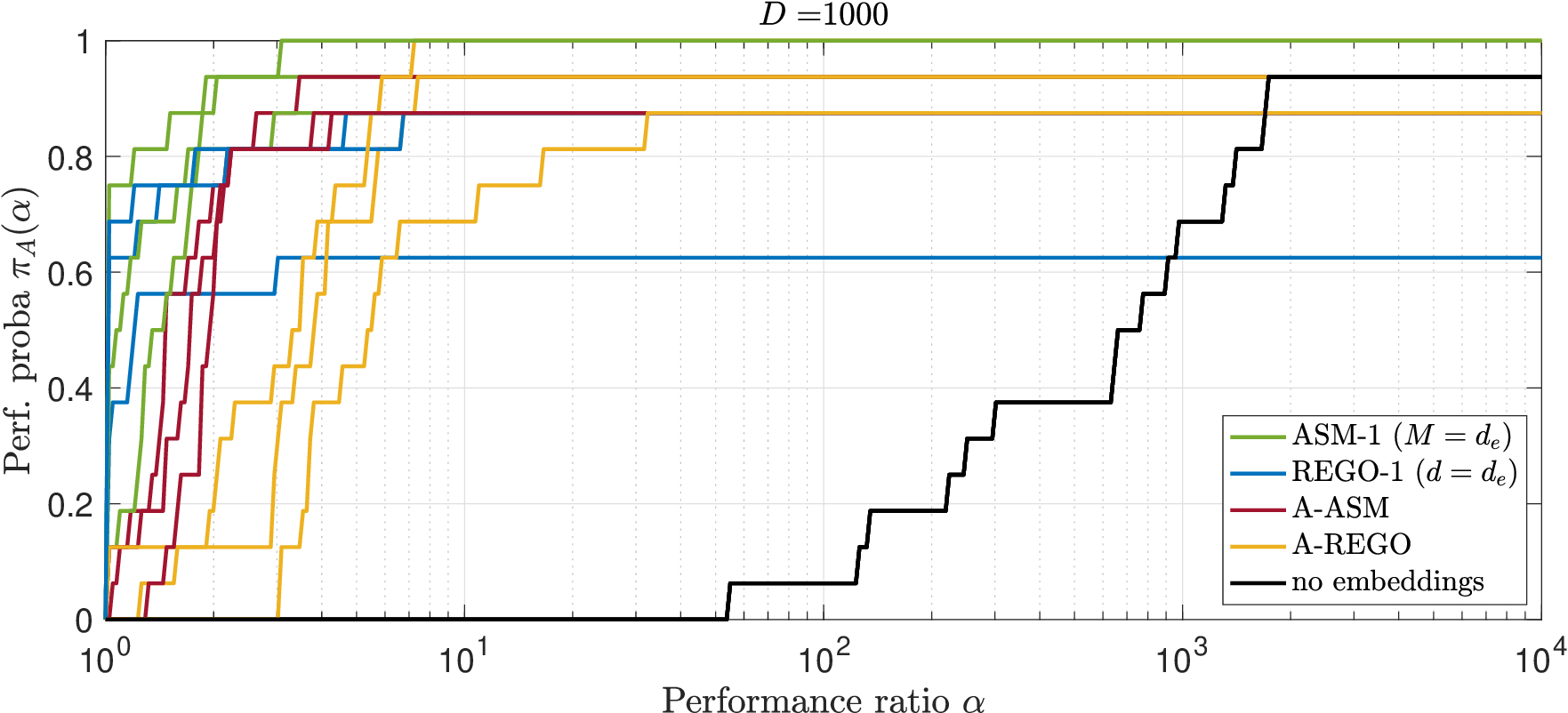}  }
\caption{Comparing ASM-1, A-ASM, REGO-1, A-REGO and no-embedding (with mKNITRO) for functions with low effective dimensionality. The performance metric  here is CPU time in seconds. Lines with the same colour represent three different realisations of an algorithm.}
\label{fig:perf_t}
\end{figure}

Table \ref{table:prob_succ_noada} lists the test functions for which at least one realization (out of three) of the solver ASM-1, REGO-1, A-ASM or A-REGO did not solve the original problem \eqref{P}.  The test problems that are not listed in this table are successfully solved by all four algorithms for $D \in \{100, 1000\}$. The last column indicates whether the no-embedding framework solved the problem (1) or not (0). Note that on these problems, ASM-1 is better than REGO-1 in terms of success rates, while A-ASM and A-REGO perform comparably.


\begin{table}[h]
\caption{The estimated effective dimension found by our adaptive algorithms for three random seeds. Numbers in bold are the exact effective dimension.}
\label{table:estimate_d}
\centering
\fontsize{7pt}{7pt}\selectfont
\setlength\tabcolsep{4 pt}
\begin{tabular}{|r|cccccccccccccccc|}
\hline
&
\raisebox{-4\normalbaselineskip}[0pt][60pt]{\rotatebox[origin=c]{90}{Beale}} &
\raisebox{-4\normalbaselineskip}[0pt][0pt]{\rotatebox[origin=c]{90}{Branin}} &
\raisebox{-4\normalbaselineskip}[0pt][0pt]{\rotatebox[origin=c]{90}{Brent}} &
\raisebox{-4\normalbaselineskip}[0pt][0pt]{\rotatebox[origin=c]{90}{Camel}} &
\raisebox{-4\normalbaselineskip}[0pt][0pt]{\rotatebox[origin=c]{90}{Goldstein-Price}} &
\raisebox{-4\normalbaselineskip}[0pt][0pt]{\rotatebox[origin=c]{90}{Hartmann 3}} &
\raisebox{-4\normalbaselineskip}[0pt][0pt]{\rotatebox[origin=c]{90}{Hartmann 6}} &
\raisebox{-4\normalbaselineskip}[0pt][0pt]{\rotatebox[origin=c]{90}{Levy}} &
\raisebox{-4\normalbaselineskip}[0pt][0pt]{\rotatebox[origin=c]{90}{Rosenbrock}} &
\raisebox{-4\normalbaselineskip}[0pt][0pt]{\rotatebox[origin=c]{90}{Shekel 5}} &
\raisebox{-4\normalbaselineskip}[0pt][0pt]{\rotatebox[origin=c]{90}{Shekel 7}} &
\raisebox{-4\normalbaselineskip}[0pt][0pt]{\rotatebox[origin=c]{90}{Shekel 10}} &
\raisebox{-4\normalbaselineskip}[0pt][0pt]{\rotatebox[origin=c]{90}{Shubert}} &
\raisebox{-4\normalbaselineskip}[0pt][0pt]{\rotatebox[origin=c]{90}{Styblinski-Tang}} &
\raisebox{-4\normalbaselineskip}[0pt][0pt]{\rotatebox[origin=c]{90}{Trid}}  &
\raisebox{-4\normalbaselineskip}[0pt][0pt]{\rotatebox[origin=c]{90}{Zettl}} \\
\hline
A-ASM ($10^2$) & {\bf 2} & {\bf 2} & {\bf 2} & {\bf 2} & {\bf 2} & {\bf3} & {\bf 6} & {\bf 6} & {\bf 7} & {\bf 4} & {\bf 4} & {\bf 4} & {\bf 2} & {\bf 8} & {\bf 5} & {\bf 2}\\
A-ASM ($10^3$) & {\bf 2} & {\bf 2} & {\bf 2} & {\bf 2} & {\bf 2} & {\bf 3} & {\bf 6} & {\bf 6} & {\bf 7} & {\bf 4} & {\bf 4} & {\bf 4} & {\bf 2} & {\bf 8} & {\bf 5} & {\bf 2}\\ 
A-REGO ($10^2$) & {\bf 2} & {\bf 2} & {\bf 2} & {\bf 2} & {\bf 2} & {\bf 3} & {\bf 6}(7,{\bf 6}) & {\bf 6}(7,{\bf 6}) & {\bf 7} & {\bf 4} & {\bf 4}(5,{\bf 4}) & {\bf 4} & {\bf 2} & 9(9,{\bf 8}) & {\bf 5} & {\bf 2} \\ 
A-REGO ($10^3$) & {\bf 2}({\bf 2},3) & {\bf 2} & {\bf 2} & {\bf 2} & {\bf 2} & {\bf 3} & {\bf 6} & {\bf 6}(7,{\bf 6}) & 8({\bf 7},8) & {\bf 4} & {\bf 4} & {\bf 4} & {\bf 2} & {\bf 8} & {\bf 5} & {\bf 2}\\ 
\hline
\end{tabular}
\end{table}

\begin{table}[h]
\caption{Problems for which at least one solver ASM-1, REGO-1, A-ASM or A-REGO did not return an optimal solution (1 means success and 0 means failure for all three random seeds).
}
\label{table:prob_succ_noada}
\centering
\fontsize{9pt}{9pt}\selectfont
\begin{tabular}{ |c|cc|cc|cc|cc|cc| }
    \hline
       &  \multicolumn{2}{c|}{ASM-1} & \multicolumn{2}{c|}{REGO-1}  & \multicolumn{2}{c|}{A-ASM}  & \multicolumn{2}{c|}{A-REGO}  &\multicolumn{2}{c|}{no-embedding}  \\ 
       $D$ & $10^2$ &  $10^3$ & $10^2$ & $10^3$ & $10^2$ & $10^3$ & $10^2$ & $10^3$ & $10^2$ & $10^3$  \\
    \hline \hline 
    Goldstein-Price & 1 & 1 & 1 & 2/3 & 1& 1& 1& 1 & 1 & 1 \\ 
    \hline
    Levy & 2/3  & 1/3  & 0 & 0 & 2/3 & 0 & 1/3 & 1/3 & 0 & 1\\ 
    \hline
    Shekel 5 & 1 & 1 & 2/3 &  2/3 & 1 & 1 & 1 & 1 & 1 & 1\\ 
    \hline
    Shekel 7 & 1 & 1 & 1 & 2/3  & 1 & 1 & 1 & 1 & 1 & 1\\ 
    \hline
    Shekel 10 & 1 & 1 & 2/3 & 2/3 & 1 & 1 & 1 & 1 & 1 & 1\\ 
    \hline
    Styblinski-Tang & 0 & 2/3 & 2/3 &  0 & 1/3 & 1/3 & 2/3 & 2/3 & 1 & 0\\ 
    \hline
\end{tabular}
\end{table}


\paragraph{Experiment 2: varying the effective dimension $d_e$.} In Experiment 1, the effective dimension $d_e$ of the different problems  ranges from 2 to 10, which is significantly smaller than the search space dimension $D \in \{100,1000\}$. Here, we compare instead A-ASM, ASM-1, A-REGO, and REGO-1 for two problems (Rosenbrock and Trid) in the regime $D = 100$, and $d_e \in \{10, 20, 50\}$.  Figures \ref{fig:change_de_Rosenbrock} and  \ref{fig:change_de_Trid} compare the number of function evaluations and CPU time (averaged over three random seeds) needed by each algorithm to reach the above-mentioned stopping criteria. Note that active subspace methods result in  lower computational costs than their random embedding counterparts, and that the gap increases with the effective dimension $d_e$.

The rate of success of the algorithms to return a solution of the original problem \eqref{P}, as well as the estimated dimension, are given in Table \ref{table:change_de_short}. Note that in most cases, the algorithms are comparable, but overall, active subspace  algorithms result in larger   success rates and more accurate estimates of the effective dimension than the random embedding  methods.

\begin{figure}[h]
\subfloat{
  \centering
  \includegraphics[width=0.45\linewidth]{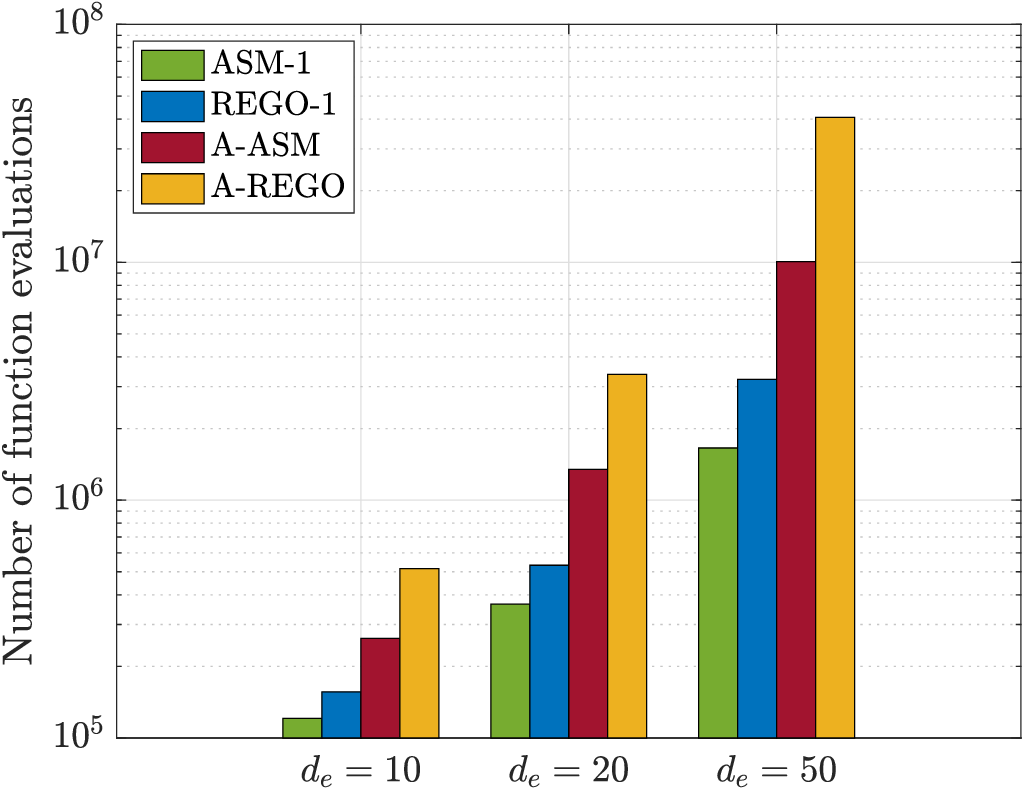}  }
\subfloat{
  \centering
  \includegraphics[width=0.45\linewidth]{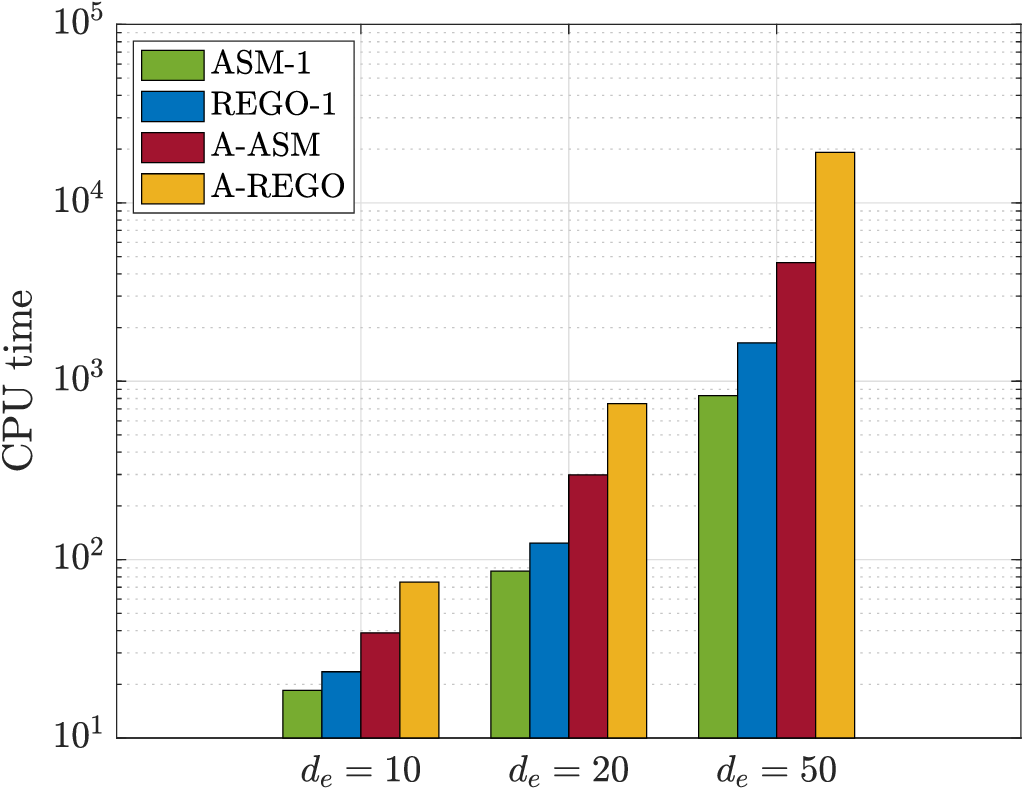} }
  \caption{Computational costs of ASM-1, A-ASM, REGO-1 and A-REGO for the Rosenbrock function with $d_e \in \{10,20,50\}$ and $D = 100$ using three random seeds.} 
\label{fig:change_de_Rosenbrock}
\end{figure}

\begin{figure}[h]
\subfloat{
\centering
  \includegraphics[width=0.45\linewidth]{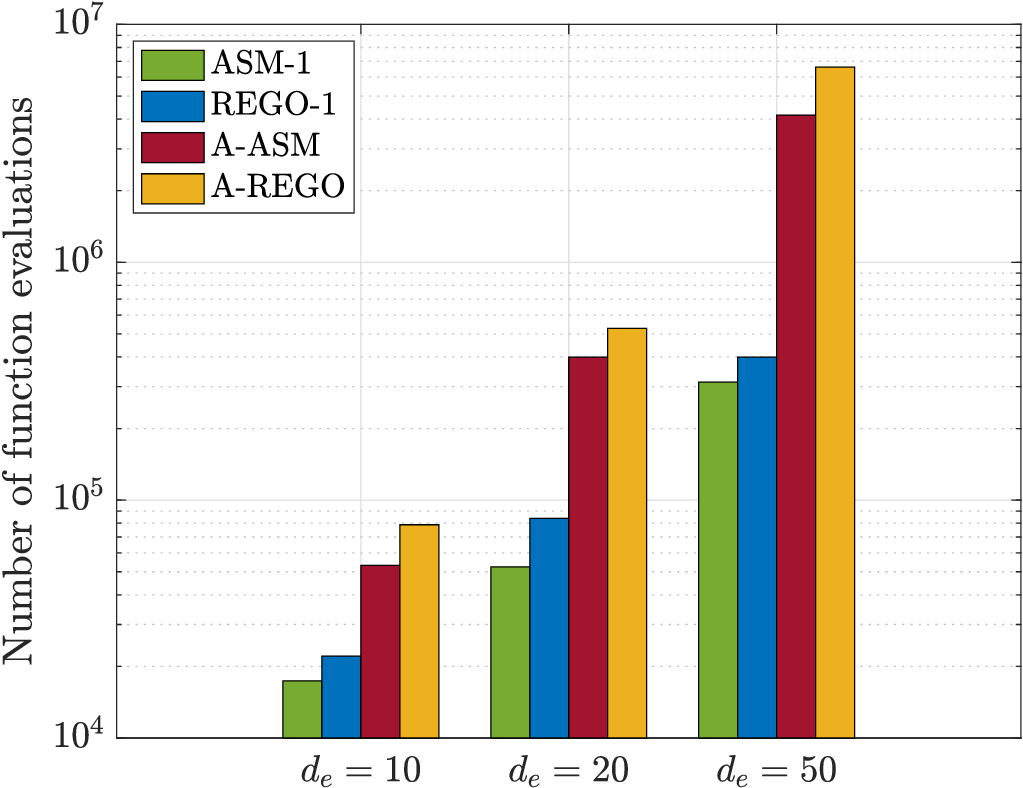} }
\subfloat{
  \centering
  \includegraphics[width=0.45\linewidth]{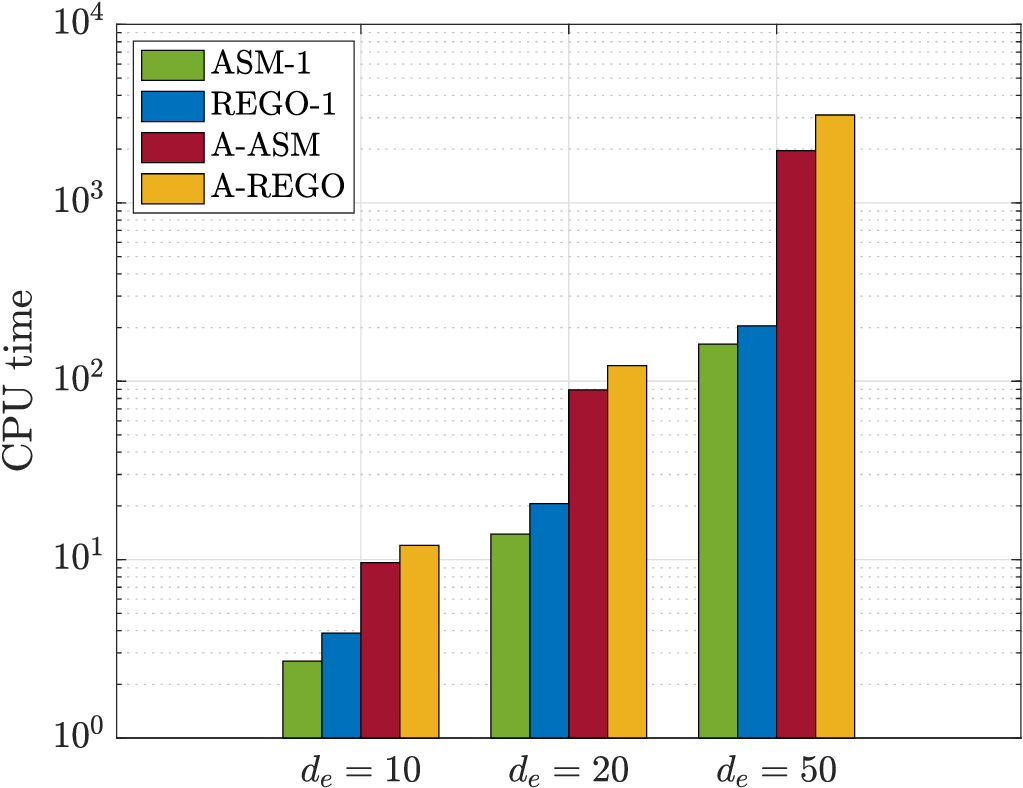}  }
\caption{Computational costs of ASM-1, A-ASM, REGO-1 and A-REGO for the Trid function with $d_e \in \{10,20,50\}$ and $D = 100$ using three random seeds.} 
\label{fig:change_de_Trid}
\end{figure}

\begin{table}[h]
\caption{Estimated effective dimension and rate of success of ASM-1, A-ASM, REGO-1 and A-REGO on the Rosenbrock and Trid functions with $d_e \in \{10,20,50\}$, averaged over three random seeds.}
\label{table:change_de_short}
\centering
\fontsize{9pt}{9pt}\selectfont
\setlength\tabcolsep{3.5 pt}
\begin{tabular}{ |c||c|c|c|c||c|c|c|c| }
    \hline
    & \multicolumn{4}{c||}{$d_{\text{est}}$} & \multicolumn{4}{c|}{Success rate over 3 seeds} \\
    & ASM-1 & REGO-1 & A-ASM & A-REGO & ASM-1 & REGO-1 & A-ASM & A-REGO\\ 
    \hline
    Ros. 10  & 10 & 10 & 10 & 10 & 1 & 1 & 1 & 1 \\ 
    \hline
    Ros. 20  & 20 & 20 & 20 & 20 & 1 & 1 & 1 & 1 \\ 
    \hline
    Ros. 50 & 50 & 50 & 50 & 50(52,50)  & 1 & 1/3 & 1 & 1 \\
    \hline
    Trid 10 & 10 & 10 & 10 & 10  & 1 & 1 & 1 & 1 \\ 
    \hline
    Trid 20 & 20 & 20 & 20 & 20  & 1 & 1 & 1 & 1 \\ 
    \hline
    Trid 50 & 50 & 50 & 50 & 52(53,50) & 1 & 1/3 & 1 & 1 \\ 
    \hline
\end{tabular}
\end{table}

\paragraph{Experiment 3: sampling complexity analysis.}
\label{sec:choose_M}
The aim of this experiment is to estimate $M$\footnote{For reference, we remind the reader here of our sampling results in Theorem  \ref{thm:estimate_on_M} and Remark \ref{remark:M}. Here, we find that numerically, in many cases, $M=d_e$ is sufficient, apart from some difficult examples.},
the number of sample gradients needed to successfully find the effective dimension (namely, the exact dimension of the active subspace) for several benchmark functions (results averaged over five random seeds). 
In this experiment, we compute the matrix $\hat \bC$ using \Cref{alg:est_C}, based on $M$ samples drawn independently at random according to a density $\rho$ chosen here as the standard Gaussian distribution in $\R^D$, and measure the effective dimension as the rank\footnote{We rely here on MATLAB built-in function \texttt{rank} with default tolerance.} of $\hat \bC$.  Figure \ref{fig:prob_succ_M} showcases $M$ for several benchmark functions (results averaged over five random seeds). 

Note that, for Rosenbrock (and most of the functions of our test set), the effective subspace dimension is found as soon as $M = d_e$, which means that the first $d_e$ sampled gradients $\nabla f(\bx_S^1), \dots, \nabla f(\bx_S^{d_e})$ are linearly independent. Thus, for all these objectives, sampling $d_e$ points is typically sufficient to find the effective dimension, as the sampled gradients are generally linearly independent. A first exception is  Hartmann3 with $D = 1000$, where the same conclusion holds for 4 out of the 5 seeds tried; in the remaining case, the gradients  $\nabla f(\bx_S^1), \dots, \nabla f(\bx_S^{d_e})$ are deemed linearly dependent, but the correct dimension is found as soon as $M=d_e+1$.
A more noticeable exception is the Easom function, due to its very sharp peak at the optimum; see \Cref{fig:alpha_easom} (left). For this function, most of the sampled gradients are located outside of the peak, hence are very close to zero. To better illustrate the impact of the geometry of the function on sampling complexity, we define a rescaled version of the Easom function, referred to as $\alpha$Easom, and defined as 
\[ f(x_1,x_2) = -\cos(\psi(x_1)) \cos(\psi(x_2)) \exp(-(\psi(x_1)-\pi)^2-(\psi(x_2)-\pi)^2),\]
with $\psi(x) = \alpha(x-\pi)+\pi$. The change of variables $\psi$ allows widening the peak of the function, as illustrated in \Cref{fig:alpha_easom} (right) for $\alpha = 0.1$, so that the difficulty to sample informative gradients is diminished. For $\alpha = 1$, $\alpha$Easom becomes the usual Easom function. 
Figure \ref{fig:prob_succ_M} illustrates the probability to find the correct effective dimension using \Cref{alg:est_C} for $\alpha$Easom with $\alpha \in \{0.1, 0.5, 1\}$. Note that, as $\alpha$ decreases, the success probability increases for a given $M$; when $\alpha=0.1$, the required number of samples to estimate $d_e$ is close to $M = d_e$, as we were observing for the other functions of our test set. A more detailed view of the impact of  scaling on sampling complexity is provided in \Cref{fig:Easom_prob_succ}, where the lowest value of $M$ for which $\hat \bC$ has rank $d_e$ is displayed for 5 different seeds; this figure confirms the trend observed in \Cref{fig:prob_succ_M} for a grid of $\alpha$ values. 

\begin{figure}
\subfloat{
  \centering
  \includegraphics[width=0.9\linewidth]{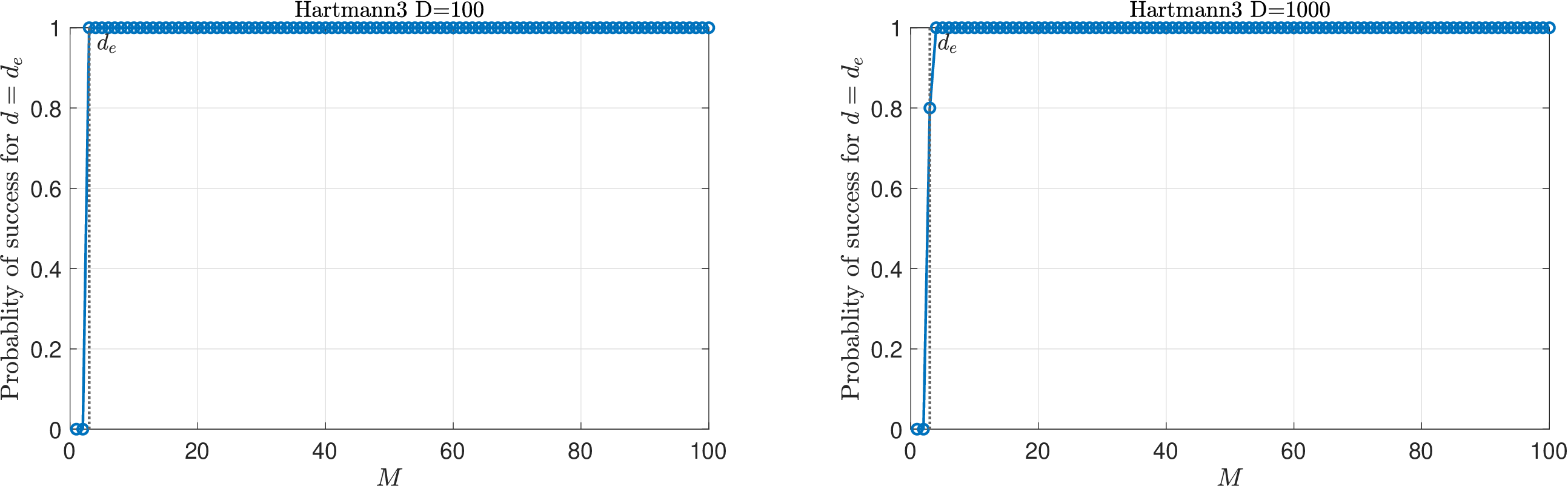}}
\newline
\subfloat{
  \centering
  \includegraphics[width=0.9\linewidth]{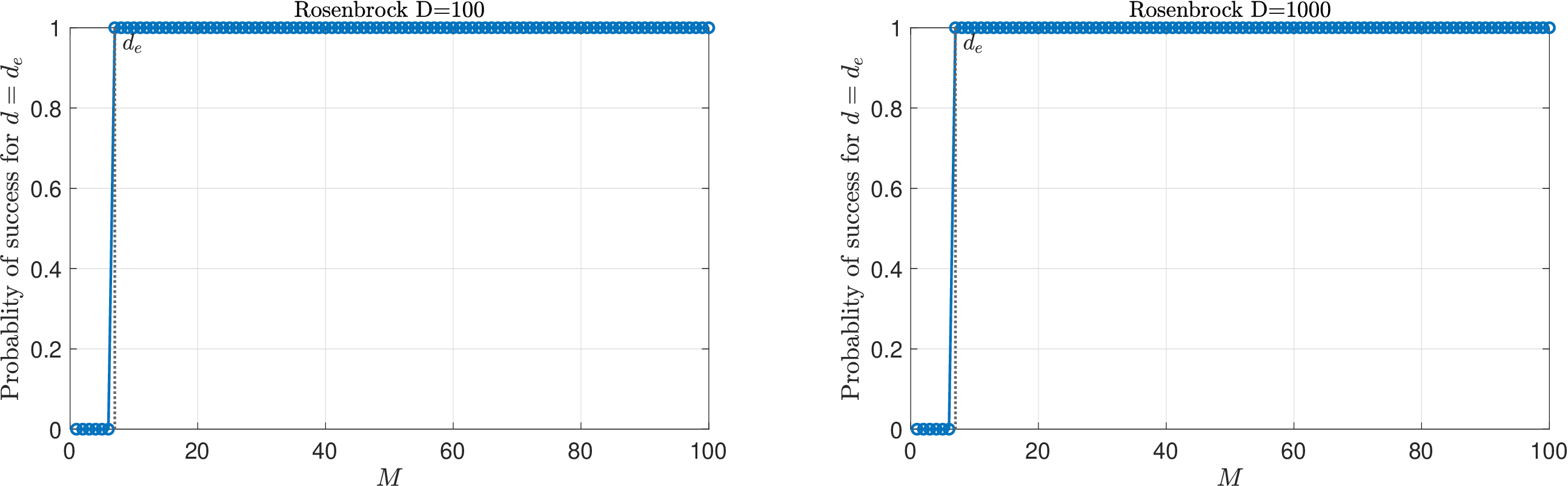}}
\newline
\subfloat{
\centering
  \includegraphics[width=0.9\linewidth]{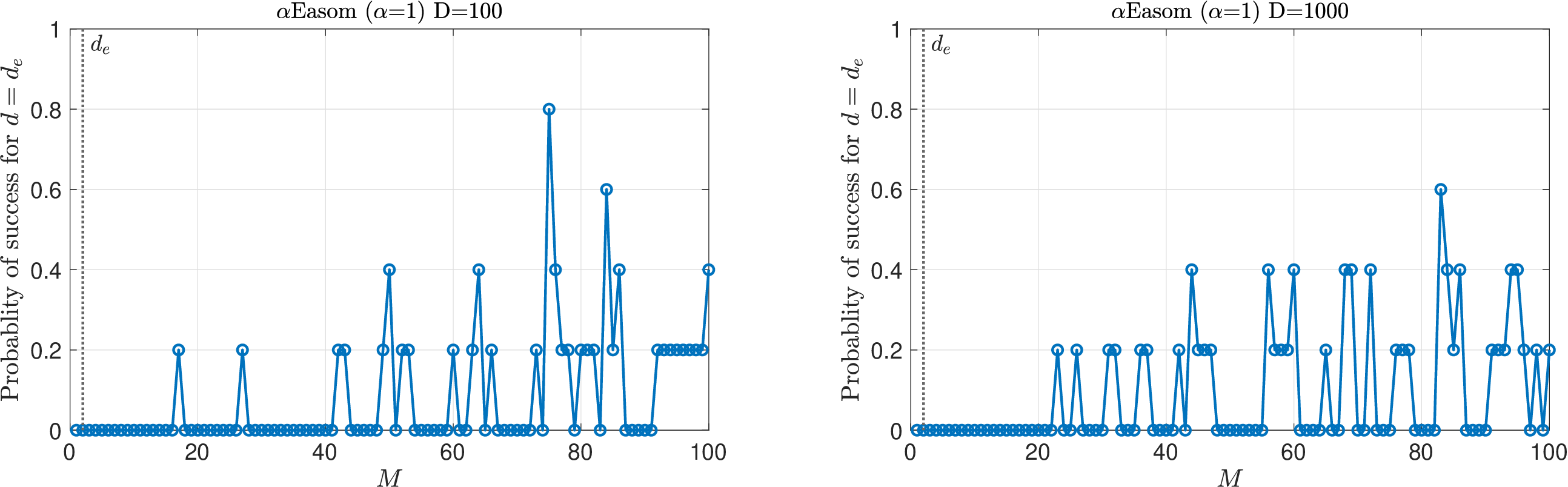}}
\newline
\subfloat{
\centering
  \includegraphics[width=0.9\linewidth]{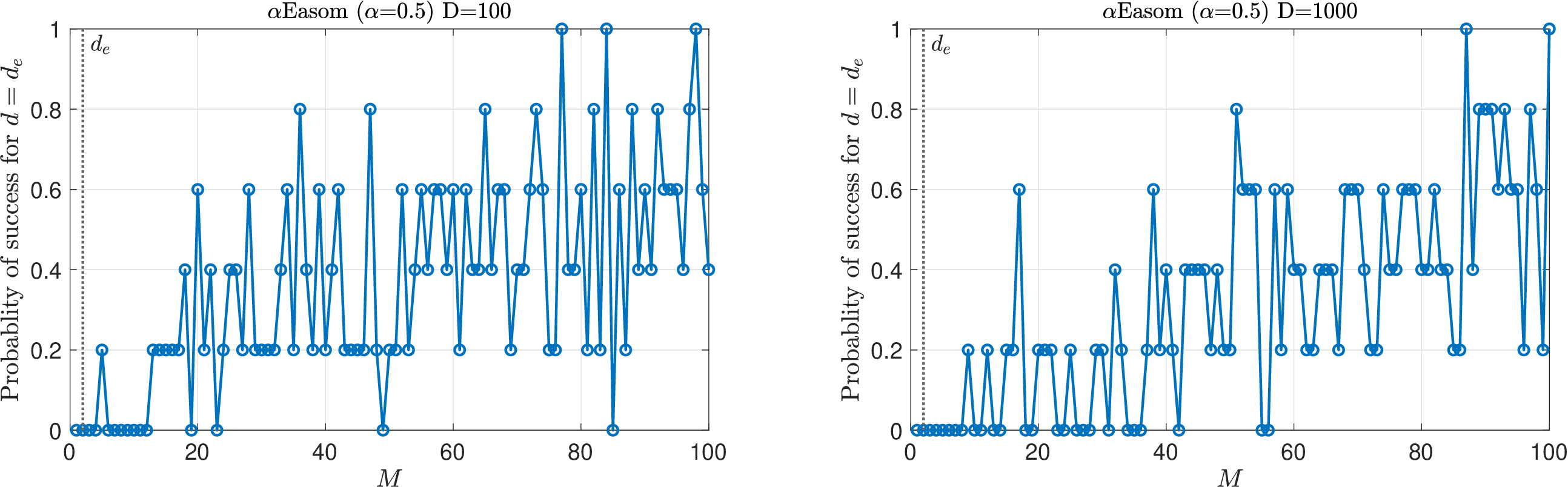}}
  \newline
\subfloat{
\centering
  \includegraphics[width=0.9\linewidth]{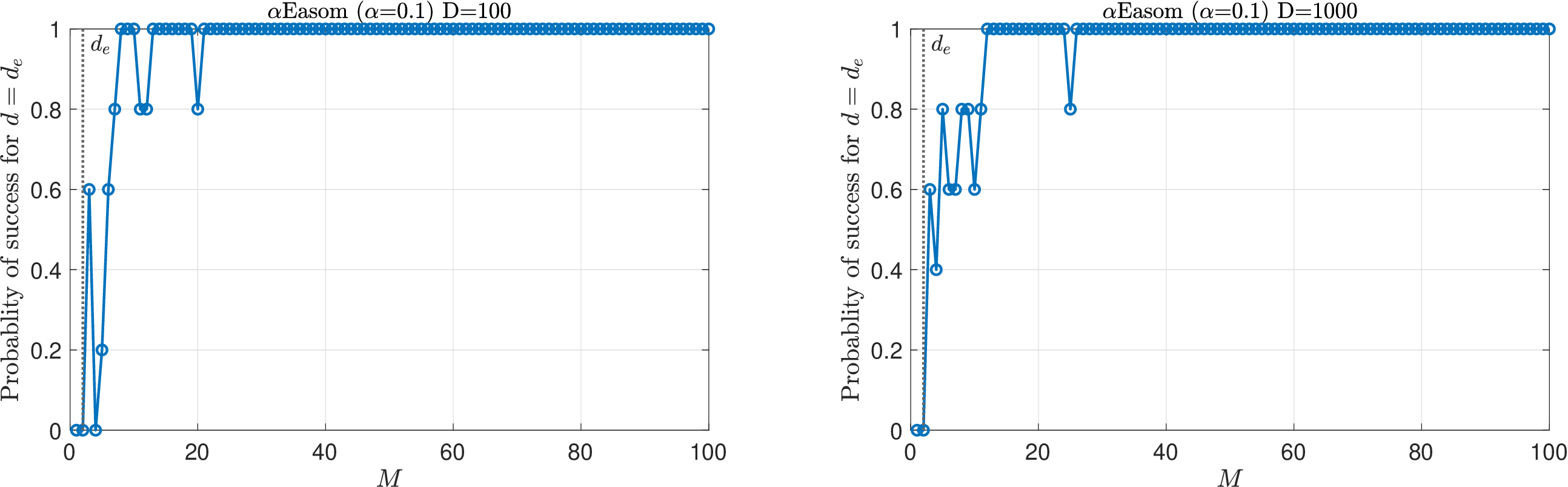}}
\caption{Probability of finding the exact dimension of the effective subspace ($d_e$) using \Cref{alg:est_C} with $M$ samples over $5$ seeds.}
\label{fig:prob_succ_M}
\end{figure}

\begin{figure}[h]
\subfloat{
    \centering
    \includegraphics[width=0.45\linewidth]{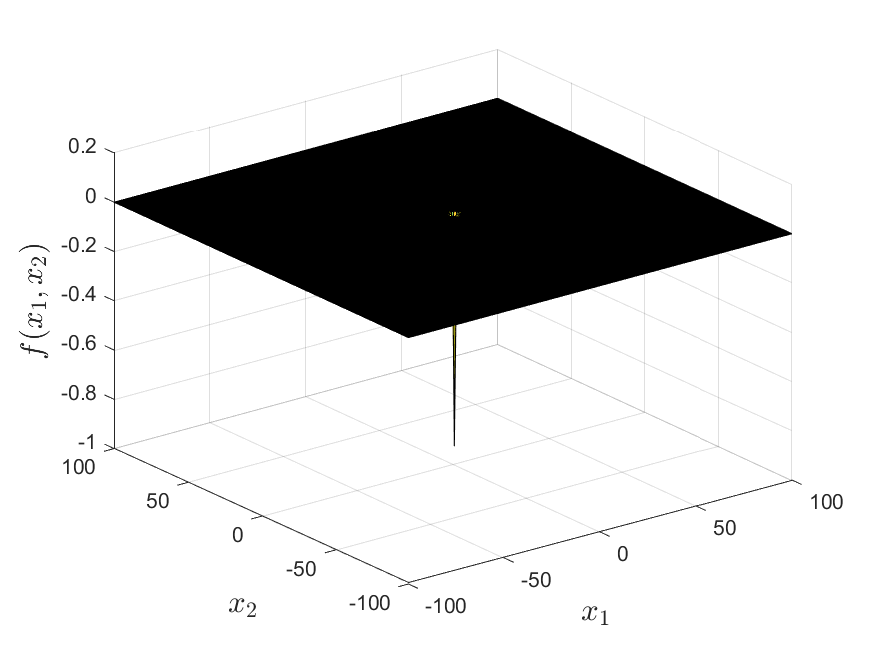}}
\subfloat{
    \centering
    \includegraphics[width=0.45\linewidth]{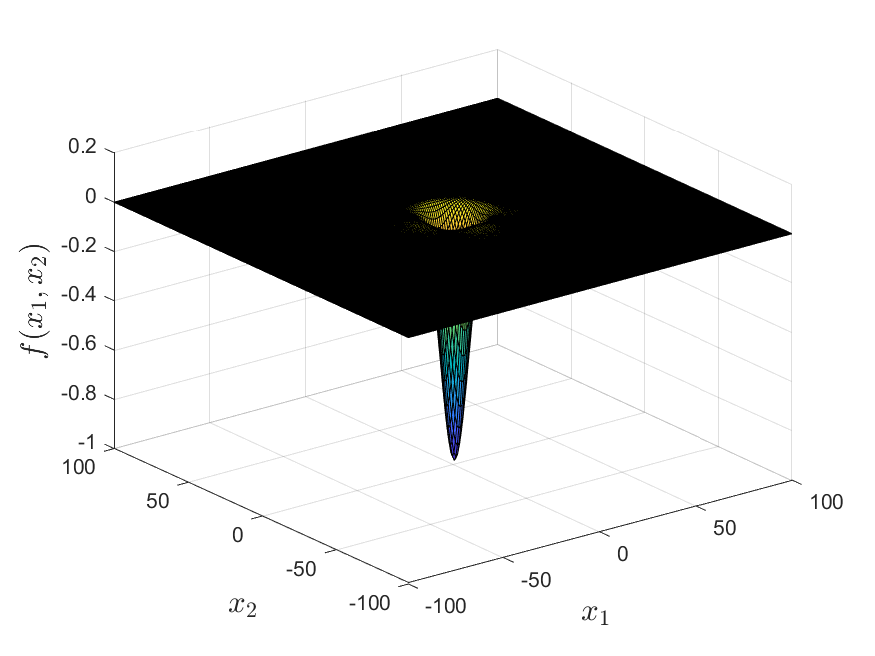}}
\caption{Original Easom function (left), and of $\alpha$Easom for $\alpha = 0.1$ (right).} 
\label{fig:alpha_easom}
\end{figure}

\begin{figure}[h!]
    \centering
    \includegraphics[width=\linewidth]
    {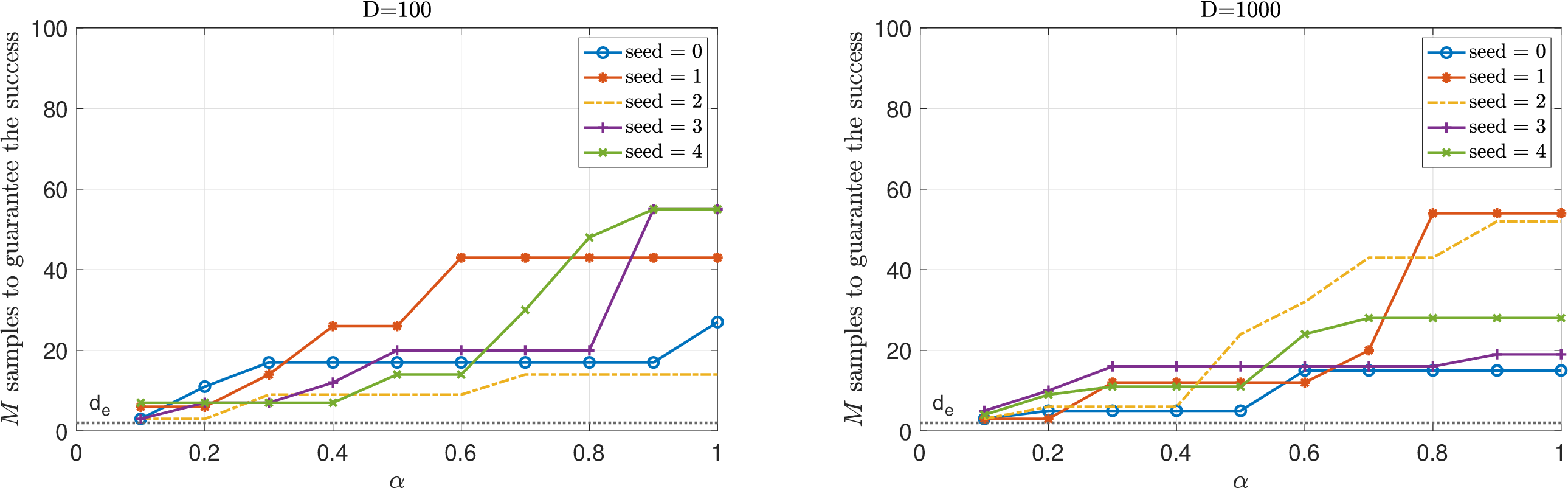}
    \caption{Minimum number of samples for finding the effective dimension ($d_e = 2$) using \Cref{alg:est_C} for 5 random seeds and for $\alpha \in \{0.1, 0.2, \dots, 1\}$.}
\label{fig:Easom_prob_succ}
\end{figure}

{\bf Concluding comments to the numerical experiments.} Overall, we found that knowing the effective dimension a priori benefits algorithm performance, both when random and active embeddings are used. When $d_e$ is unknown, the adaptive variant that learns the subspace of variation is at least as good as its random embedding counterpart. In the adaptive active subspace algorithms, we have only grown the subspace/sampling by one (gradient) vector on each (major) iteration; this is similar to the adaptive random embeddings framework where also, the random subspace increases slowly. Alternative approaches could be developed that sample more than one gradient at each iteration, reducing the number of subspace subproblems that need solving. The precise trade-off between sample complexity and computational cost of the subproblems remains to be decided.



\bibliography{refs.bib}
\bibliographystyle{plain}

\appendix
\newpage
\section{Proofs of some results in Section \ref{sec: algo}}

\paragraph{Proof of Lemma \ref{lemma: lim_of_prob_is_1}.}
We define an auxiliary random variable
\[ J^K :=  \mathds{1} \left(\bigcup_{k=1}^K \left\{  \{R^k = 1\} \cap \{ S^k = 1 \} \right\} \right) =  1- \prod_{k=1}^{K} (1-R^k S^k).\]
Then,
\begin{align*}
	\prob\Big[ \bigcup_{k=1}^K \left\{ \{R^k = 1\} \cap \{ S^k = 1 \} \right\} \Big] &= \mathbb{E}[J^K]  = 1 - \mathbb{E}\Big[\prod_{k=1}^{K} (1-R^k S^k)\Big] \\
	& \stackrel{(*)}{=} 1 - \mathbb{E}\Big[\mathbb{E}\Big[ \prod_{k=1}^{K} (1-R^k S^k) \Big| \mathcal{F}^{K-1/2} \Big]\Big] \\
	& \stackrel{(\circ)}{=} 1 - \mathbb{E}\Big[ \prod_{k=1}^{K-1} (1-R^k S^k) \cdot \mathbb{E}\big[ 1 - R^K S^K | \mathcal{F}^{K-1/2} \big]\Big] \\
	& \geq 1 - \mathbb{E}\Big[(1-\gamma R^K)\cdot \prod_{k=1}^{K-1} (1-R^k S^k)  \Big] \\
	& \stackrel{(*)}{=} 1 - \mathbb{E}\Big[\mathbb{E}\Big[(1-\gamma R^K)\cdot \prod_{k=1}^{K-1} (1-R^k S^k)   \Big| \mathcal{F}^{K-1}\Big]\Big] \\
	& \stackrel{(\circ)}{=} 1 - \mathbb{E}\Big[ \prod_{k=1}^{K-1} (1-R^k S^k) \cdot \mathbb{E}\big[ 1 - \gamma R^K | \mathcal{F}^{K-1} \big]\Big] \\
	& \geq 1 - (1-\tau \gamma) \cdot \mathbb{E}\Big[ \prod_{k=1}^{K-1} (1-R^kS^k)\Big],
	\end{align*}
where 
\begin{itemize}
\item[-]
$(*)$ follow from the tower property of conditional expectation (see (4.1.5) in \cite{Durrett2019}), 
\item[-]
$(\circ)$ is due to the fact that $R^1, \dots, R^{K-1}$ and $S^1,\dots,S^{K-1}$ are $\mathcal{F}^{K-1/2}$-  \,and $\mathcal{F}^{K-1}$-measurable (see Theorem 4.1.14 in \cite{Durrett2019}), 
\item[-]
 the inequalities follow from \Cref{corr:success_solver} and \Cref{corr: lowerbdRPK}, respectively. 
\end{itemize}
We repeatedly expand the expectation of the product for $K-1$, $\ldots$, $M_0$, in exactly the same manner as above, to obtain 
\begin{align*}
    \prob\Big[ \bigcup_{k=1}^K \left\{ \{R^k = 1\} \cap \{ S^k = 1 \} \right\} \Big] & \geq 1 - (1-\tau \gamma)^{K-M_0+1} \cdot \mathbb{E}\Big[ \prod_{k=1}^{M_0 - 1} (1-R^kS^k)\Big] \\
    & \geq 1 - (1-\tau\gamma)^{K-M_0+1}, 
\end{align*}
where the last inequality follows from $\mathbb{E}\Big[ \prod_{k=1}^{M_0 - 1} (1-R^kS^k)\Big] \leq 1$ which is true since $(1-R^kS^k) \leq 1$ for $k\geq 1$.

\paragraph{Proof of Lemma \ref{lemma: if WcapG then_x in G_epsilon}.}
By \Cref{def:RPsuccessful}, if \eqref{RP_k} is successful, then there exists $\mvec{y}^k_{int} \in \mathbb{R}^d$ such that
	\begin{equation} \label{ineq: asym_conv_ineq1}
	f(\mtx{A}^k\mvec{y}^k_{int} + \mvec{p}^{k-1}) = f^*.
	\end{equation}
	Since $ f^k_{min}$ is the global minimum of \eqref{RP_k}, we have
	\begin{equation} \label{ineq: asym_conv_ineq3}
	 f^k_{min} \leq f(\mtx{A}^k\mvec{y}^k_{int} + \mvec{p}^{k-1}).
	\end{equation}
	Then, for $\mvec{x}^k$, \eqref{l_f0} gives the first inequality below,
	$$ f(\mvec{x}^k) \leq f^k_{min} + \epsilon \leq f(\mtx{A}^k\mvec{y}^k_{int} + \mvec{p}^{k-1}) + \epsilon = f^* + \epsilon, $$
	where the second inequality is due to \eqref{ineq: asym_conv_ineq3} and the equality follows from \eqref{ineq: asym_conv_ineq1}. This shows that $\mvec{x}^k \in G_\epsilon$. 

\section{Test problem generation}
\label{appen: table}

Table~\ref{table:test_set} contains the name, domain and global minimum of our benchmark functions, selected from \cite{adorio2005mvf, gavana2016global, bingham2013virtual}.

\begin{table}[h]
\caption{Our benchmark function set}
\centering
\begin{tabular}{ |c|c|c| }
    \hline
      Function &  Domain  & Global minimum \\ 
    \hline
     Beale~\cite{adorio2005mvf} & $\by \in [-4.5, 4.5]^{2}$ & $h(\by^{*})=0$ \\ 
    \hline
     Branin~\cite{adorio2005mvf} & \makecell{$\by_{1} \in [-5, 10]$ \\ $\by_{2} \in [0, 15]$} & $h(\by^{*})=0.397887$ \\ 
   \hline
    Brent~\cite{gavana2016global} & $\by \in [-10, 10]^{2}$ & $h(\by^{*})=0$ \\ 
    \hline
    Camel~\cite{adorio2005mvf} & \makecell{$\by_{1} \in [-3, 3]$ \\ $\by_{2} \in [-2, 2]$} & $h(\by^{*})=-1.0316$ \\ 
    \hline
    Goldstein-Price~\cite{adorio2005mvf} & $\by \in [-2, 2]^{2}$ & $h(\by^{*})=3$ \\ 
    \hline
    Hartmann 3~\cite{adorio2005mvf} & $\by \in [0, 1]^{3}$ & $h(\by^{*})=-3.86278$ \\ 
    \hline
    Hartmann 6~\cite{adorio2005mvf} & $\by \in [0, 1]^{6}$ & $h(\by^{*})=-3.32237$ \\ 
    \hline
    Levy~\cite{bingham2013virtual} & $\by \in [-10, 10]^{6}$ & $h(\by^{*})=0$ \\ 
    \hline
    Rosenbrock~\cite{bingham2013virtual}  & $\by \in [-5, 10]^{7}$ & $h(\by^{*})=0$ \\ 
    \hline
    Shekel 5~\cite{bingham2013virtual}  & $\by \in [0, 10]^{4}$ & $h(\by^{*})=-10.1532$ \\ 
    \hline
    Shekel 7~\cite{bingham2013virtual}  & $\by \in [0, 10]^{4}$ & $h(\by^{*})=-10.4029$ \\ 
    \hline
    Shekel 10~\cite{bingham2013virtual}  & $\by \in [0, 10]^{4}$ & $h(\by^{*})=-10.5364$ \\ 
    \hline
    Shubert~\cite{bingham2013virtual}  & $\by \in [-10, 10]^{2}$ & $h(\by^{*})=-186.7309$ \\ 
    \hline
    Styblinski-Tang~\cite{bingham2013virtual}  & $\by \in [-5, 5]^{8}$ & $h(\by^{*})=-313.329$ \\ 
    \hline
    Trid~\cite{bingham2013virtual}  & $\by \in [-25, 25]^{5}$ & $h(\by^{*})=-30$ \\ 
    \hline
    Zettl~\cite{adorio2005mvf} & $\by \in [-5, 5]^{2}$ & $h(\by^{*})=-0.00379$\\ 
    \hline 
\end{tabular}
\label{table:test_set}
\end{table}

Each function $h$ in Table~\ref{table:test_set} is turned into a function $f$ over $\R^D$, for $D \in \{100,1000\}$. We proceed in three steps. First, we apply a suitably chosen linear change of variables to each function $h$ of \Cref{table:test_set}, to turn it into a function $\bar h$ whose domain is $[-1,1]^{d_{e}}$. Then, we add $D-d_{e}$ fake dimensions in the search space, with zero coefficient:
\begin{align*}
	h(\bx) = \bar{h}(\bar{\bx})+0\cdot x_{d_{e}+1}+0\cdot x_{d_{e}+2}+ \cdots+0\cdot x_{D}.
\end{align*}
Finally, in order to obtain a non-trivial constant subspace, we rotate the function $h(\bx)$ by applying a random orthogonal matrix $\bQ$ to $\bx$. The $D$-dimensional function we used in our test is then given by 
\begin{align*}
	f(\bx) = h(\bQ\bx).	
\end{align*}
Note that, by construction, the first $d_{e}$ rows of $\bQ$ form a basis of the effective subspace $\bT$ of $f$, and the last $D-d_{e}$ rows  of $\bQ$ form a basis of the constant subspace $\bT^{\bot}$ of $f$.\\

\section{Random embedding algorithms}
\label{appen: alg}

We recall here the A-REGO (see \cite{cartis2023generalf}) and REGO-1 algorithms (Algorithm REGO in \cite{cartisOtemissov2022}).

\begin{algorithm}[H]
\begin{algorithmic}[1]
\State Give an initial dimension $d^{1}=1$ and $\bp^{0} \in \bR^{D}$
\For{$k \geq 1$ until termination criterion is satisfied}
\State Generate a $D \times d^{k}$ Gaussian matrix $\bA^{k}$
\State Calculate $\by^{k}$ by applying a global optimization solver to
\begin{equation}
\begin{aligned}
	& \min_{\by \in \bR^{d^{k}}} f(\bA^{k} \by + \bp^{k-1}) 
\end{aligned}
\label{RP_r}
\end{equation}
\State Construct $\bx^{k}=\bA^{k}\by^{k}+\bp^{k-1}$
\State Choose $\bp^{k} \in \bR^{D}$ (deterministically or stochastically) and $d^{k+1}=d^{k}+1$
\EndFor
\end{algorithmic}
\caption{A-REGO applied to~(\ref{P}) for $f$ with low effective dimensionality}
\label{REGO}
\end{algorithm}

\begin{algorithm}[H]
\begin{algorithmic}[1]
\State Give the dimension $d$ and $\bp \in \bR^{D}$ \State Generate a $D \times d$ Gaussian matrix $\bA$
\State Calculate $\by$ by applying a global optimization solver to
\begin{equation}
\begin{aligned}
	& \min_{\by \in \bR^{d}} f(\bA \by + \bp)  
\end{aligned}
\label{RP_r}
\end{equation}
\State Let $\bx=\bA \by + \bp$.
\end{algorithmic}
\caption{REGO-1 applied to~(\ref{P}) for $f$ with low effective dimensionality}
\label{REGO-1}
\end{algorithm}

\end{document}